\documentclass[12pt,oneside,epsfig,amsfonts]{amsart}
\usepackage[latin9]{inputenc}
\usepackage[a4paper]{geometry}
\usepackage{amstext}
\usepackage{amsthm}
\usepackage{amssymb}
\usepackage{graphicx}
\usepackage{xcolor}

\makeatletter
\numberwithin{equation}{section}
\numberwithin{figure}{section}
\theoremstyle{plain}
\newtheorem{thm}{\protect\theoremname}
  \theoremstyle{plain}
  \newtheorem{lem}[thm]{\protect\lemmaname}
  \theoremstyle{plain}
  \newtheorem{cor}[thm]{\protect\corollaryname}
  \theoremstyle{plain}
  \newtheorem{prop}[thm]{\protect\propositionname}

\usepackage{epsfig}
\usepackage{eufrak}
\usepackage{xcolor}
\makeatother

  \providecommand{\corollaryname}{Corollary}
  \providecommand{\lemmaname}{Lemma}
  \providecommand{\propositionname}{Proposition}
\providecommand{\theoremname}{Theorem}

\begin{document}
\begin{abstract}
Infectious diseases are among the most prominent threats to mankind. When preventive health care cannot be provided, a viable means
of disease control is the isolation of individuals, who may be infected.
To study the impact of isolation, we propose a system of Delay Differential
Equations and offer our model analysis based on the geometric theory
of semi-flows. Calibrating the response to an outbreak in terms of 
the fraction of infectious individuals isolated and the speed with which this is done, 
we deduce the minimum response required to curb an incipient 
outbreak, and predict the ensuing endemic state should the infection continue to spread.
\end{abstract}

\title{SIQ: a delay differential equations model for\\
 disease control via isolation}

\date{August 20, 2017}

\author{Stefan Ruschel\textsuperscript{1,2}, Tiago Pereira\textsuperscript{2}, Serhiy Yanchuk\textsuperscript{1}, Lai-Sang Young\textsuperscript{3}}

\maketitle

\textsuperscript{1}Institut für Mathematik, Technische Universität
Berlin,

\textsuperscript{2}Instituto de Ciências Matemáticas e de Computação,
Universidade de São Paulo,

\textsuperscript{3}Courant Institiute of Mathematical Sciences, New
York University.

\medskip

\section{Introduction}

%

In the recent outbreaks of swine flu, Sars, bird flu, and Ebola, local health authorities were not prepared to deal with the developing crisis. Reasons vary.  
In the case of Ebola, it took a while to recognize the urgency of the situation and the affected countries lacked the needed infrastructure.
In the case of Sars, the means of transmission was unknown and a vaccine 
was not available. In these situations and others, 
health authorities have recommended the {\it isolation of individuals}, who may be infected \cite{CDC_legal, Siegel2007, centers2014}.
This is only natural: in the absence of other means to curb 
the spreading of a disease, the only way to slow down its propagation is to deny possible infection pathways. Strategies of this kind date back several centuries and 
their usefulness has not diminished with time, as evidenced in recent events \cite{kucharski2015, Donnelly2003}.

In any isolation strategy, early identification of infectious individuals is crucial.
It is also a formidable task. 
Adequate infrastructure  and constant preparedness is costly to maintain; 
infected individuals themselves may fail to recognize the 
potential danger they pose to others, for reasons of their own 
some may choose not to seek medical attention; and coercive measures
can be controversial.
For these and other reasons, it is important for health authorities to properly evaluate 
in advance
the level of response capabilities needed to combat outbreaks, to determine what fraction of the infectious population must be identified, by which means and how quickly \cite{Day2006, Fraser2004, Peak2017}. 
The optimal duration of isolation
is another question not well understood. Can, for example, longer isolation
compensate for slower identification?

While statistics have been collected and analyzed for a number of specific diseases,
the impact of isolation, 
in particular the human toll caused by failures or delays in its implementation, has not received a great deal of attention \cite{Zuzek2015,Pereira2015}.
These papers used different models to shed light
on the relation between network structure, isolation, and propagation rate, relying on the theory of branching processes to approximate early phases of the infection.
The nonlinear effects of isolation and the prediction of the endemic state when isolation fails were beyond the scopes of these earlier studies.  

This paper contains a theoretical study of the use of isolation to control 
the spreading of infectious diseases, focusing on the consequences of imperfect 
implementation such as failure to identify a fraction of 
the infected hosts and
delays in isolating them from the general public. 
Without limiting ourselves to specific diseases, 
we deduce, based on general disease reproductive characteristics, the minimum 
response required to curb a developing epidemic. When this minimum response 
is not met and the infection becomes endemic, we offer predictions on 
the fraction of the population that can be expected 
to fall ill. We believe an improved understanding of  issues of this kind 
will be of use to health authorities as they assess the costs and
benefits of their policies.

Our study is carried out using a dynamical systems approach. The theory of
nonlinear dynamical systems permits us both to carry out local, linear analyses
and to use global, geometric techniques to study the nonlinear effects of isolation 
and its impact on the eventual endemic state. We started from a network 
in which each node represents
an individual. Under some simplifying assumptions, we derive a system of Delay Differential Equations describing the time course of an infection following an outbreak. This system of differential equations give rise to an infinite dimensional dynamical system that, as we will show, is amenable to detailed mathematical analysis. 
Throughout the paper we give broad biological interpretations of our findings and support them with 
technical results that we believe are of independent mathematical interest.




\section{ Model description}

We study an extension of the SIS (susceptible-infectious-susceptible) model with the additional feature that a fraction of the infectious individuals will be isolated. Consider, to begin with, a network of $N$ nodes; each node represents a host, and nodes that are linked by edges are neighbors. Each host has two discrete states: healthy and susceptible ($S$), and infectious ($I$). Infected hosts infect their neighbors until they recover and rejoin the susceptible group. Models of this type have been studied a great deal and require no further introduction. We refer the reader to Refs.~\cite{Anderson1991,Diekmann2000,Brauer2008,Keeling2011}
for a broad introduction.

In this work, we consider a model as above with the additional feature of isolation of infected hosts. Specifically, if a host remains infectious for $\tau$ units of time without having recovered, it enters a new state, $Q$ (for isolation or \textit{quarantine}) with probability $p$. 

We are aware that the term ``quarantine'' in the literature refers to the isolation of individuals who may be infected but are not yet symptomatic \cite{Siegel2007,centers2014}. The letter "Q" here, is solely used to clearly distinguish it from the infectious class $I$.

The hosts that do not enter state $Q$ at time $\tau$ remain infectious until they recover on their own. 
A host that enters state $Q$ remains in this state for $\kappa$ units of time, at the end of which it is discharged and rejoins the healthy and susceptible pool. We define $r$ to be the reproductive number of the disease in the absence of isolation, i.e. for $p=0$. Note that this deviates from the canonical choice of the capital letter $R_0$, which we will use for the reproductive number of the disease including isolation, i.e. when $p>0$.

The numbers $r,\tau,\kappa>0$ and $p\in[0,1]$ are to be viewed as parameters of the model, with $\tau$ representing the {\it identification time} between the infection and isolation, and $\kappa$ the {\it isolation time}. The number $p$ can be interpreted as the probability of an infectious host being diagnosed and isolated, we call it {\it identification probability}. Table \ref{Tb1} summarizes the main parameters of the SIQ model and their meaning. See Fig.~\ref{fig:model} for a schematic of the model.

We now go to a mean field approximation of this process. Let $S(t),I(t)$ and $Q(t)$ denote the fractions of individuals in the corresponding states at
time $t$, so that $S(t)+I(t)+Q(t)=1$ and the size of the population is assumed 
to be constant. Assuming the independence of the susceptible and
infectious groups, we arrive at the following system of delay differential equations:
\begin{eqnarray}
\dot{S}\negmedspace\left(t\right) & = & -rS\negmedspace\left(t\right)I\negmedspace\left(t\right)+I\negmedspace\left(t\right)+r\varepsilon S\negmedspace\left(t-\tau-\kappa\right)I\negmedspace\left(t-\tau-\kappa\right),\label{eq:S-dyn}\\
\dot{I}\negmedspace\left(t\right) & = & rS\negmedspace\left(t\right)I\negmedspace\left(t\right)-I\negmedspace\left(t\right)-r\varepsilon S\negmedspace\left(t-\tau\right)I\negmedspace\left(t-\tau\right),\label{eq:I-dyn}\\
\dot{Q}\negmedspace\left(t\right) & = & r\varepsilon\left[S\negmedspace\left(t-\tau\right)I\negmedspace\left(t-\tau\right)-S\negmedspace\left(t-\tau-\kappa\right)I\negmedspace\left(t-
\tau-\kappa\right)\right],\label{eq:Q-dyn}
\end{eqnarray} 
where $\varepsilon:=pe^{-\tau}$ can be interpreted as the effectiveness
of the identification process. Detailed explanations of the modeling leading to system \eqref{eq:S-dyn}\textendash \eqref{eq:Q-dyn} are given in the Appendix. 

\begin{table}[htbp]
\centering
\begin{tabular}{c|l} 
\hline
parameter & meaning \\
\hline 
\hline
$r$ & reproductive number of the disease in the absence of isolation ($p=0$)\\ 
$p$ & probability to identify an infectious individual \\ 
$\tau$ & time elapsed between infection and identification \\
$\kappa$ & time spent in isolation after identification \\
$\varepsilon$ &  effectiveness of the identification process $(\varepsilon=pe^{-\tau})$\\
\hline
 \end{tabular}
  \vspace{0.2cm}
   \caption{{\small Main parameters of the SIQ model.}}
   \label{Tb1}
\end{table}

\begin{figure}[h]
\includegraphics[width=\linewidth]{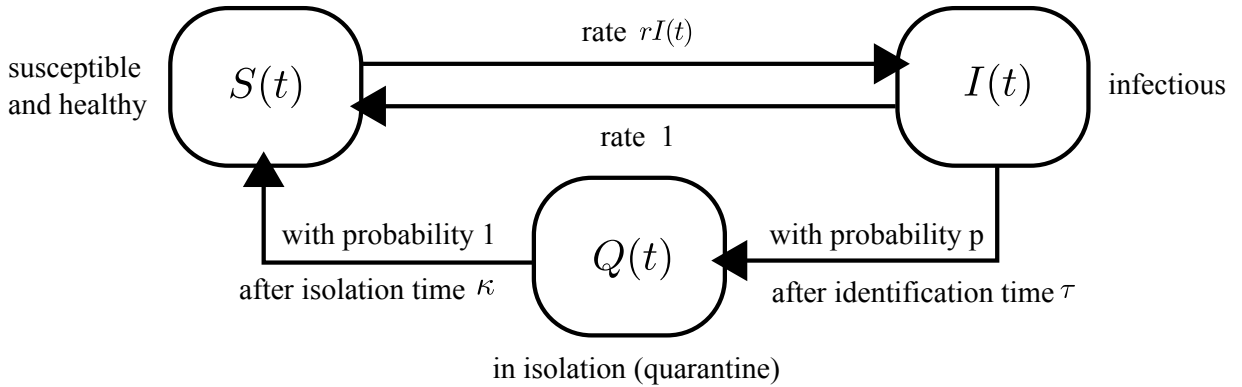}%
\caption{\label{fig:model}Illustration of the SIQ model. 
The resulting SIQ model is an extension of the SIS model with the additional feature that with probability
$p$ individuals that have been infectious for a time $\tau$ are identified and isolated for $\kappa$ units of time at the end of which they are healthy again.}
\end{figure}

This model neglects several aspects of epidemic scenarios, such as the acquisition of immunity or delays in the development of infectiousness. To demonstrate that the model described above is generalizable, we will, 
in Sec. 8, introduce an {\it latency period} $\sigma$ to become infectious after being infected to the model above, and show how much
of the analysis carries over. For conceptual clarity, we will first treat the $\sigma=0$ 
case in Secs. 3--7.


\section{Non-technical overview of the main results}
\label{sec:overview}

In this section, we describe the main results leaving precise
technical formulations for later sections.  Recall that without the isolation strategy our SIQ model reduces to the SIS model with disease reproduction
number $r$, so that an infection spreads if and only if $r>1$. 
Of interest in this paper is the case $r>1$, so that
if no measures are taken the infection will spread. 
Consider a history $(\phi_S(t), \phi_I(t), \phi_Q(t)), t \le 0$,
corresponding to the sudden appearance of a small infection at time $t=0$.
For definiteness, let $\phi_Q(t)=0$ for all $t \le 0$, $\phi_I(t) =0$ for all $t<0$, $0<\phi_I(0) \ll 1$, and $\phi_S+\phi_I+\phi_Q \equiv 1$.
Unless otherwise stated, this history will be assumed in the discussion below.
Our main results can be summarized as follows:

\medskip
\begin{itemize}
\item[1.] {\it Required minimum identification probability.} We prove that 
an outbreak can be prevented only if
\begin{equation}\label{eq:pc}
p>p_{c}=1-1/r\ ,
\end{equation}
that is, to have a chance to stop the outbreak, one must be able to identify 
a sufficiently large fraction of infectious individuals.

\medskip
\item[2.] {\it Critical identification time.}
Possessing the ability to detect individuals with probability   $p > p_c$ alone  is not enough; one must be prepared to act with sufficient
speed: we prove that for each $p>p_c$, there is a critical identification time 
\begin{equation}\label{eq_tauc}
\tau_{c}(p) = \ln \frac{p}{p_c}.
\end{equation}
Specifically, for $p>p_c$ and $\tau<\tau_{c}$, the infection dies out. 
In this case, the time $\kappa$ that infectious individuals spend in isolation is
of no consequence.  These results are presented in Sec.~\ref{sec:Neighborhood-of-Disease-Free}.
\end{itemize}

\medskip
 We can readily compute  the critical identification time $\tau_{c} $ for various diseases once we have the reproductive number $r$ and identification probability $p$. 
In Eqs. (\ref{eq:S-dyn})-(\ref{eq:S-dyn}), we have done the usual rescaling 
$t \mapsto  t / \gamma$ where $\gamma$ is the rate of recovery
(see the Appendix for a full discussion). This means that $\tau_c$ is also rescaled. 
While that is convenient mathematically, it is also interesting to compare critical identification times {\it without rescaling},
so that we can analyze diseases in their natural time spans. To that end, 
we define
\begin{equation}\label{eq:Tc}
T_c = \frac{\tau_c}{\gamma},
\end{equation}
and show, in Table \ref{Tb2}, the critical response capability $p_c$ and critical identification time $T_c$ for $p=0.8$. 

\begin{table}[htbp]
\centering
\begin{tabular}{c|ccccc} 
\hline
& $r$ &  $1/\gamma$ & $p_c$ & $T_c$ \\
\hline
\hline
H1N1 2016 [Brazil] \cite{WHO_flunet, Muller2015}  &  1.7  & 7.0 & 0.41 & 4.7 \\
Ebola 2014 [Guin./Lib.]  \cite{Althaus2014} & 1.5 & 12.0 & 0.33 & 10.5\\
Ebola 2014 [Sierra Leone] \cite{Althaus2014} & 2.5 & 12.0 & 0.6 & 3.5\\
Spanish Flu 1917 \cite{Muller2015}  & 2 & 7.0 &   0.5 & 3.3 \\
Influenza A  \cite{Muller2015} & 1.54 & 3.0 & 0.35 & 1.0 \\
Hepatitis A \cite{Peak2017} & 2.25 & 13.4 & 0.56 & 4.89 \\
SARS \cite{Peak2017} & 2.90 & 11.8 & 0.66 & 4.31 \\
Pertussis \cite{Peak2017} & 4.75 & 68.5 & 0.79 & 0.91 \\
Smallpox \cite{Peak2017} & 4.75 & 17.0 & 0.79 & 0.26 \\ 
\hline
 \end{tabular}
 \vspace{0.2cm}
   \caption{{\small Critical response capability $p_c$ and critical identification time $T_c$ (in days) for various diseases with basic reproductive number $r$ as well as $1/\gamma$ (in days). The critical $\tau_c$ (before rescaling) is calculated using \eqref{eq:pc}--\eqref{eq:Tc} assuming that $80\%$ of infectious individuals are identified and isolated. The values of $r$ and $1/\gamma$ are taken from the references given in the first row.
}}
\label{Tb2}
\end{table}

As shown in Table \ref{Tb2}, even when the fraction of identified individuals is as high as $80\%$ the critical identification time $T_c$ can be as 
short as $3$ days for severe outbreaks such as the Spanish Flu and the Ebola in Sierra Leone. Of major concern is what happens if such an
identification time is not met. Our next result addresses this scenario. 

\medskip
\begin{itemize}
\item[3.] {\it Prediction of endemic state as function of $\varepsilon$ and $\kappa$.}
 From Items 1 and 2, we know that when $p<p_c$ or $\tau> \tau_c$, 
so that $\varepsilon = pe^{-\tau} < 1- \frac1r$,
the infection will persist. When that happens, we prove that if the system tends to an endemic equilibrium,
the fraction of infectious individuals in the endemic state will be 
$$
I = \frac{1}{1 - \varepsilon + \varepsilon \kappa}  \cdot \left(1-\varepsilon-\frac{1}{r}\right).
$$
Notice that increasing $\kappa$ leads to an endemic equilibrium with a smaller $I$.
\end{itemize}

\medskip
As an illustration consider a hypothetical response to the Ebola  outbreak in Sierra Leone with $\varepsilon = 0.5$. We obtain that the final fraction of infectious individuals in the endemic state is $I =  0.2/({1 + \kappa})$.

\medskip
\begin{itemize}
\item[4.] {\it Bifurcation analysis at endemic equlibria.} 
For each $p$ and $\tau$ with $\tau \ll1$, we performed 
a rigorous bifurcation analysis at each endemic
equilibrium point with $\kappa$ as bifurcation 
parameter. We proved that the equilibrium destabilizes through a Hopf
bifurcation as $\kappa$ is increased, and that it undergoes
a cascade of Hopf bifurcations as $\kappa$ is increased further.

\bigskip
\item[5.] {\it Effect of $\kappa$ on the course
of an epidemic.} Item 4 described the dynamics near an endemic
equilibrium irrespective of how we got there. Here we return to the setting
of Item 3, i.e. the sudden appearance of a small infection that gets out of control, and ask how the duration of isolation will influence the course of events. Our results for this part are numerical.
We show that the infection will approach the endemic equilibrium predicted
in Item 3, and that as $\kappa$ increases, the equilibrium destabilizes 
through a Hopf bifurcation in a manner similar to that described in Item 4. 
For large $\kappa$, our simulations suggest that the fraction of infectious individuals, can have periodic oscillations with nontrivial amplitudes. 
These results are presented in Sec.~\ref{sec:Case-of-endemic-infection}.
\end{itemize}

\medskip
To summarize, the SIQ model offers quantitative measures
for critical response capabilities and identification times needed to prevent
outbreaks of infectious diseases. For endemic infections, our analysis
offers guidance to optimal choices of isolation durations. The implications
of these results on epidemics control are clear: Isolation of infectious
hosts is not without cost, both in terms of society and economics.
These must be weighed against the costs of an endemic infection, as
well as strategies for disease management. The SIQ model proposed here
may assist in such costs-and-benefits analysis.


\section{Basic Properties of the Model}

\subsection{Mathematical framework}

Equations \eqref{eq:S-dyn}\textendash \eqref{eq:Q-dyn} 
define a dynamical system on the phase space 
$C:=C\left(\left[-\tau-\kappa,0\right],\mathbb{R}^{3}\right)$,
the Banach space of continuous functions with the norm 
\[
\left\Vert \phi\right\Vert =\sup_{\theta\in\left[-\tau-\kappa,0\right]}\left|\phi\left(\theta\right)\right|\ ,
\]
$|\cdot|$ being the Euclidean norm in $\mathbb{R}^{3}$. Given an
initial function $\phi\in C$, the solution $x(t,\phi),$ $t\ge0$,
of the initial value problem to \eqref{eq:S-dyn}\textendash \eqref{eq:Q-dyn}
exists and is unique \cite{Hale1993}. We use the standard notation 
\[
x_{t}(\phi)=x(t+\theta;\phi),\quad\theta\in[-\tau-\kappa,0].
\]
This solution defines a $C^{1}$ semiflow $T^{t}:\phi\mapsto x_{t}(\phi)$
on $C$ \cite{Hale1993}. 

Observe that the conservation of mass property of Eqs.~\eqref{eq:S-dyn}\textendash \eqref{eq:Q-dyn},
namely $S'(t)+I'(t)+Q'(t)\equiv0$, implies that if $\phi=(\phi_{S},\phi_{I},\phi_{Q})$
and $x(t;\phi)=(S(t),I(t),Q(t))$, then $S(t)+I(t)+Q(t)=\phi_{S}(0)+\phi_{I}(0)+\phi_{Q}(0)$
for all $t\ge0$. In particular, the manifold 
\[
\mathcal{C}:=\{\phi\in C([-\tau-\kappa,0],\mathbb{R}^{3})\  | \ \phi_{S}(\theta)+\phi_{I}(\theta)+\phi_{Q}(\theta)=1\mbox{ for all }\theta\in[-\tau-\kappa,0]\}
\]
is positively invariant with respect to the semiflow $T^{t}$.

In the context of our epidemiological model, all solutions of interest
have the property that for each $t$, $x(t;\phi)$ takes value in
the $2$-simplex 
\[
\Delta^{2}=\left\{ u=(u_{1},u_{2},u_{3})\in\mathbb{R}^{3}:\sum_{i}u_{i}=1,u_{i}\ge0,i=1,2,3\right\} ,
\]
i.e., $x_{t}(\phi)\in\tilde{\mathcal{C}}=\{\psi\in\mathcal{C}:\psi(\theta)\in\Delta^{2}\mbox{ for all }\theta\in\left[-\tau-\kappa,0\right]\}$
for all $t\ge0$. In Sec. \ref{subsec:Biologically-relevant-solutions}
we show that biologically relevant initial conditions that belong
to a certain subset of $\tilde{\mathcal{C}}$ lead to solutions that
belong to $\tilde{\mathcal{C}}$ for all $t\ge0$. When studying $T^{t}$
as a dynamical system, it is conceptually simpler to work with $\mathcal{C}$
as the phase space. We will therefore do that in our theoretical investigations,
and focus on trajectories with $T^{t}(\phi)\in\tilde{\mathcal{C}}$
in biological interpretations. Observe that the dynamics on $\mathcal{C}$
are completely determined by any two of Eqs.~\eqref{eq:S-dyn}\textendash \eqref{eq:Q-dyn}
together with the conservation of mass.


\subsection{Equilibrium solutions and $\omega$-limit sets\label{subsec:Equilibrium}}

Recall that the \textit{$\omega$-limit set} $\omega(\phi)$
of $\phi\in\mathcal{C}$ under the semi-flow $T^{t}$ is defined to
be 
\[
\omega(\phi)\ =\ \{\psi\in\mathcal{C}\ | \ T^{t_{n}}\phi\to\psi\mbox{ for some sequence }t_{n}\to\infty\}.
\]
For a solution that is bounded, $x_{t}(\phi)$ is $C^{1}$ with a
uniform bound on its derivatives for all $t\ge\tau+\kappa$. Thus,
by the Arzela-Ascoli Theorem, $\omega(\phi)$ is nonempty and compact
in $\mathcal{C}$ (with its $C^{0}$ norm).

In particular, consider an equilibrium solution  $\phi$ of Eqs.~\eqref{eq:S-dyn}\textendash \eqref{eq:Q-dyn}, which means that $x\left(t;\phi\right)=\phi\left(0\right)$ for all $t$, and $\phi$ is a constant function.
For $u=(u_{1},u_{2},u_{3})\in\mathbb{R}^{3}$ with $\sum_{i}u_{i}=1$,
we will use the notation $\hat{u}$ to denote the constant function
in $\mathcal{C}$ with $\hat{u}(\theta)=u$ for all $\theta$. The
equilibria of Eqs.~\eqref{eq:S-dyn}\textendash \eqref{eq:Q-dyn}
can be computed as follows. If $\hat{\phi}=(\hat{\phi}_{S},\hat{\phi}_{I},\hat{\phi}_{Q})$
is an equilibrium, then it must satisfy
\begin{equation}
0=\left(r\left(1-\varepsilon\right)\phi_{S}-1\right)\phi_{I}.\label{eq:I-dyn-2}
\end{equation}
Thus any $\hat{\phi}$ with $\phi_{I}=0$ or $\phi_{S}=1/r(1-\varepsilon)$
is an equilibrium solution. We define
\[
\mathcal{\mathcal{E}}_{0}:=\left\{ \hat{\phi}\in\mathcal{C}|\,\phi_{I}=0\right\} \qquad\mbox{and}\qquad\tilde{\mathcal{E}}_{0}:=\mathcal{\mathcal{E}}_{0}\bigcap\tilde{\mathcal{C}}
\]
to be the sets of \textit{disease-free equilibria}. Analogously we
define the sets
\[
\mathcal{E}_{I}:=\left\{ \hat{\phi}\in\mathcal{C}|\,\phi_{S}=\frac{1}{r\left(1-\varepsilon\right)}\right\} \qquad\mbox{and}\qquad\tilde{\mathcal{E}}_{I}:=\mathcal{E}_{I}\bigcap\tilde{\mathcal{C}},
\]
which we refer to as \emph{endemic equilibria} in the case $\phi_{I}\ne0$.
For $\phi\in\mathcal{C}$, it is possible that $T^{t}(\phi)$ will
approach one of the equilibria above as $t\to\infty$, but this need
not be the only possible long-time behavior (and we will show that
it is not). 

\subsection{Biologically relevant solutions and their positivity\label{subsec:Biologically-relevant-solutions}}

We consider a solution $x(t;\phi)$ as biologically relevant if $x_{t}(\phi)\in\tilde{\mathcal{C}}$
for all $t\ge0$. In this section we give sufficient conditions for
positivity. More specifically, we show that any initial condition
corresponding to an infection that started just prior to $t=0$ 
leads to a biologically relevant
solution. 

We start with a function $\psi:\left[-\tau-\kappa,0\right]\to\mathbb{R}^{3}$
(that may or may not be in $\mathcal{C}$). Think of it as a situation
we find ourselves in \textendash{} without knowledge of how we got
there. From this initial condition, we evolve the system according
to Eqs.~\eqref{eq:S-dyn}\textendash \eqref{eq:Q-dyn}. The next
lemma gives conditions on $\psi$ that will lead to biological solutions. 
\begin{lem}
\label{lem:positivity}Let $\psi=(\psi_{S},\psi_{I},\psi_{Q}):[-\tau-\kappa,0]\to\mathbb{R}^{3}$
be a piecewise continuous function with values in $\Delta^2$.
We assume further that 
\[
\psi_{I}(0)\ge rp\int_{-\tau}^{0}e^{\theta}\psi_{S}(\theta)\psi_{I}(\theta)d\theta,
\]
and 
\[
\psi_{Q}(0)\ge r\varepsilon\int_{-\tau-\kappa}^{0}\psi_{S}(\theta)\psi_{I}(\theta)d\theta.
\]
Then $S\left(t\right),I\left(t\right),Q\left(t\right)\ge0$, \textup{$S\left(t\right)+I\left(t\right)+Q\left(t\right)=1$}
for all $t\geq0$, and $x_{t}(\psi)\in\tilde{\mathcal{C}}$ for all
$t\ge\tau+\kappa$. In particular, if $\psi\in\tilde{\mathcal{C}}$,
then $x_{t}(\psi)\in\tilde{\mathcal{C}}$ for all $t\ge0.$ 
\end{lem}

The proof of Lemma \ref{lem:positivity} follows from the derivation
in the Appendix. Note that the conditions of 
Lemma \ref{lem:positivity} are satisfied for open sets of initial conditions
corresponding to the sudden uptick of an infection  around
time $0$, described by $\psi$ with $\psi_I(0), \psi_Q(0)>0$ and $0<\psi_I(t), \psi_Q(t)
\ll \psi_I(0), \psi_Q(0)$ for all $t<-\delta$ for some $\delta>0$
sufficiently small.

\section{Neighborhood of Disease-Free Equilibria\label{sec:Neighborhood-of-Disease-Free}}

In Secs.~\ref{subsec:Linear-analysis} and \ref{subsec:The-nonlinear-picture},
we fix $r,$ $p$, $\tau$, and give a complete description of the
dynamics in a neighborhood of $\mathcal{E}_{0}$, the set of disease-free
equilibria identified in Sec.~\ref{subsec:Equilibrium}. The truly
pertinent question, however, is what $\tau$ and $p$ need to be
to curb the propagation of small initial infections for a disease
the intrinsic reproductive number of which is $r$. These questions
will be answered in Sec.~\ref{subsec:Critical-values-of}, using
the results from the first two subsections.

\subsection{Linear analysis at $\mathcal{E}_{0}$\label{subsec:Linear-analysis}}

We parametrize $\mathcal{E}_{0}$ by $ \widehat{u(q)},$
$q\in\mathbb{R},$ where $u(q)=\left(1-q,0,q\right)$, and
study the linearized equation at each point. The following Lemma gives
the characteristic equation for a general equilibrium. 
\begin{lem}
\label{lem:characteristicf} Let  $\hat{w}\in\mathcal{C}$ with $w=\left(w_{S},w_{I},1-w_{S}-w_{I}\right)$
be an equilibrium solution of Eqs.~\eqref{eq:S-dyn}\textendash \eqref{eq:Q-dyn}.
Then the characteristic equation at $\hat{w}$ is given by
\begin{equation}
\chi\left(\lambda,\hat{w}\right)=0, \label{eq:spec}
\end{equation}
where
\begin{align}
\chi\left(\lambda,\hat{w}\right)=&\lambda\left(\lambda+1-rw_{S}\left(1-\varepsilon e^{-\tau\lambda}\right)+rw_{I}\left(1-\varepsilon e^{-(\tau+\kappa)\lambda}\right)\right) \nonumber \\
&+rw_{I}\varepsilon e^{-\tau\lambda}\left(1-e^{-\lambda\kappa}\right).\nonumber
\end{align}
\end{lem}
Since $\mathcal{E}_{0}$ consists of  a line of equilibria, $0$
is clearly an eigenvalue for $w_{I}=0$ corresponding to the direction
along the line. The stability of these equilibria in
directions transverse to $\mathcal{E}_{0}$ is determined by the
remaining eigenvalues.
\begin{thm}
\label{thm:E0}Let $\tau$ and $p$ be fixed, and assume $\varepsilon=pe^{-\tau}<1$.
We denote
\[
q_{c}:=1-\frac{1}{r}\frac{1}{1-\varepsilon}.
\]
If \textup{$q\ge q_{c}$}, then $\widehat{u(q)}$ is linearly
stable, and if \textup{$q < q_{c}$}, then  $\widehat{u(q)}$  is linearly
unstable. In more detail, at $q\ne q_{c}$,
the eigenvalue $\lambda=0$ of the equilibrium $\widehat{u(q)}$  has
multiplicity $1$, and there is no other eigenvalue on the imaginary
axis. For $q\ge q_{c}$, all nonzero eigenvalues $\lambda$ have
$Re(\lambda)<0$. For $q<q_{c}$, there is exactly one eigenvalue
$\lambda_{1}$ with $Re(\lambda_{1})>0$. 
\end{thm}
\begin{proof}
Let $ \widehat{u(q)} \in\mathcal{E}_{0}.$ Using Lemma~\ref{lem:characteristicf},
the characteristic equation has the form 
\begin{equation}
\lambda\left(\lambda+1-r(1-q)\left(1-\varepsilon e^{-\tau\lambda}\right)\right)=0.\label{eq:I=00003D00003D0-hyperbolic_spec}
\end{equation}
The eigenvalue $\lambda=0$ of the first factor corresponds to the tangential direction
along the manifold $\mathcal{E}_{0}$ and the corresponding normal
eigenvalues are remaining solutions of Eq.~\eqref{eq:I=00003D00003D0-hyperbolic_spec}. The second factor has a solution $\lambda=0$, if and only if $q=q_c$ and it is easy to show that this root is simple. Next we show that $q=q_c$ is the only value for which an equilibrium can have a normal eigenvalue $\lambda$ with $Re(\lambda)=0$. The algebraic bifurcation condition $\lambda=i\omega$ implies 
\begin{equation}
\omega^{2}=\left(r\varepsilon(1-q)\right)^{2}-(1-r(1-q))^{2},\label{eq:I=00003D00003D0-hyperbolic_spec_Mod}
\end{equation}
\begin{equation}
\frac{\omega}{1-r(1-q)}=\tan\left(\omega\tau\right).\label{eq:I=00003D00003D0-hyperbolic_spec_Im}
\end{equation}
Equation \eqref{eq:I=00003D00003D0-hyperbolic_spec_Mod} admits solutions
$\omega^{2}>0$ if $q\in\left(q_{c},q_{+}\right)$, where $q_+=1-\frac{1}{r}\frac{1}{1+\varepsilon}$. Note that the
right hand side of Eq.~\eqref{eq:I=00003D00003D0-hyperbolic_spec_Mod}
attains its global maximum for 
\[
q_{\max}=1-\frac{1}{r}\frac{1}{1-\varepsilon^{2}}
\]
with the corresponding 
\[
\omega_{\max}^{2}=\frac{\varepsilon^{2}}{1-\varepsilon^{2}}.
\]
which satisfies $\left|\omega_{\max}\right|<\frac{\pi}{\tau}$ and
thus, we restrict to $\omega\in\left(-\frac{\pi}{\tau},\frac{\pi}{\tau}\right)$.
It follows from \eqref{eq:I=00003D00003D0-hyperbolic_spec_Im} that
$\omega=0$ is the only possible solution if and only if
\begin{equation}
\frac{1}{1-r(1-q)}\le\tau,\label{eq:cond}
\end{equation}
since, in this case, the function $\tan\left(\omega\tau\right)-\omega/\left(1-r(1-q)\right)$
is strictly monotone. In fact, straightforward computation shows that
Eq.~\eqref{eq:cond} is satisfied for all $q\in\left(q_{-},q_{+}\right)$,
$\tau\geq0$ and $q\in\left[0,1\right]$. Thus, Eqs.~\eqref{eq:I=00003D00003D0-hyperbolic_spec_Mod}
and \eqref{eq:I=00003D00003D0-hyperbolic_spec_Im} do not admit solutions
with $\omega>0$ and consequently, there are no further bifurcations
possible. In particular, there are no Hopf-bifurcations. 

Next we show that $Re\left(\lambda\right)<0$ for all nontrivial eigenvalues
of all $q\geq q_{c}$. We choose $q=1-\frac{1}{r}>q_{c}$,
then \eqref{eq:I=00003D00003D0-hyperbolic_spec} takes the form $\lambda+\varepsilon e^{-\tau\lambda}=0$,
where $\varepsilon\le1$. The latter equation only attains solutions
with $Re\left(\lambda\right)<0$, see e.g. \cite{HalSmith2011}. Due
to continuity, we have $Re\left(\lambda\right)<0$ for all nontrivial
eigenvalues for all $q\in\left(q_{c},1\right]\cap\left[0,1\right]$.

For any $q\in[0,q_{c})\cap\left[0,1\right]$, there is exactly
one real positive eigenvalue. Indeed, for $q=q_{c}$, the eigenvalue
crosses the imaginary axis transversely at $\lambda=0$ with the corresponding derivative 
\[
\left.\frac{\partial\left(Re\left(\lambda\right)\right)}{\partial q}\right|_{\lambda=0,q=q_{c}}<0.
\]
\end{proof}

In the context of the epidemic model, of interest is $\tilde{\mathcal{E}}_{0}\subset\tilde{\mathcal{C}}$.
We observe that $\hat{u}(q_{c})$ may or may not lie in $\tilde{\mathcal{E}}_{0}$.
In particular, if $q_{c}\le0$, then all equilibria in $\tilde{\mathcal{E}}_{0}$
are linearly stable. 
\begin{cor}\label{cor:tauc}The hypothesis are as in Theorem \ref{thm:E0}. Then the disease-free equilibrium $\widehat u(q)$ is linearly stable if $\tau$ satisfies the inequality
\begin{equation}
\tau\le\tau_c(p,q):=\ln p-\ln\left(1-\frac{1}{r(1-q)}\right).\label{eq:tauc}
\end{equation}
Otherwise, it is linearly unstable. 
\end{cor}

\subsection{The nonlinear picture near $\mathcal{E}_{0}$ \label{subsec:The-nonlinear-picture}}

As the semi-flow $T^{t}$ is $C^{1}$ (Sec. \ref{subsec:Equilibrium}),
we may appeal to invariant manifolds theory. The next theorem follows
immediately from results in \cite{Bates2000}. 
\begin{thm}
\label{thm:E0nonl}The hypotheses are as in Theorem~\ref{thm:E0}.
Then the following holds:

\begin{enumerate}
\item[(1)] Through every $\widehat{ u(q)}$ with $q>q_{c}$ passes a codimension
$1$ stable manifold $W^{s}(\widehat{ u(q)})$, with uniform estimates
away from $\widehat u(q_c)$. These manifolds foliate a uniform size
neighborhood of any compact $K\subset\{\widehat{ u(q)},q>q_{c}\}$. 
\item[(2)] Through every $\widehat{ u(q)}$ with $q<q_{c}$ passes a codimension
$2$ stable manifold $W^{s}(\widehat{ u(q)})$ and a $1$-dimensional
unstable manifold $W^{u}(\widehat{ u(q)})$, with uniform estimates away
from $\widehat{ u(q_c)}$. 
\end{enumerate}
\end{thm}
We remark that the $W^{s}$- and $W^{u}$-manifolds above are strong
stable and unstable manifolds, i.e., there exist $c=c(q)$ and $\lambda=\lambda(q)>0$
such that 
\[
\zeta\in W^{s}(\hat{u}(q))\implies\|T^{t}(\zeta)-\hat{u}(q)\|<ce^{-\lambda t}
\]
for all $t\ge0$. The dynamical picture can therefore be summarized
as follows: We partition $\mathcal{E}_{0}$ into 
\[
\mathcal{E}_{0}=\mathcal{E}_{0}^{u}\cup\mathcal{E}_{0}^{c}\cup\mathcal{E}_{0}^{s},
\]
where 
\[
\mathcal{E}_{0}^{u}=\{\widehat{u(q)},q<q_{c}\}\ ,\quad\mathcal{E}_{0}^{c}=\{\widehat{u(q_{c})}\}\ ,\quad\mbox{and}\quad\mathcal{E}_{0}^{s}=\{\widehat{u(q)},q>q_{c}\}\ .
\]

For $\phi$ sufficiently near $\mathcal{E}_{0}^{s}$, $I(t)\to0$
exponentially fast, i.e., the infection dies out quickly; while for
$\phi$ sufficiently near $\mathcal{E}_{0}^{u}$, unless $\phi$ lies
in the codimensional $1$ submanifold $\cup_{q}W^{s}(\widehat{ u(q)})$,
$I(t)$ will increase, i.e., the infection will spread, beyond a level
depending on the distance of $\phi$ to $\mathcal{E}_{0}^{c}$.


\subsection{Critical values of $p$ and $\tau$: scalings and
biological implications\label{subsec:Critical-values-of}}

We can think of $\tau$, the time between infection and isolation,
as  {\it identification time}, and $p$, the probability of an infectious
host to be properly identified and put into isolation, as \textit{isolation
probability}. With these interpretations, a question of practical
importance is the following: When presented with a scenario in which
a small fraction of the population is infectious, i.e., given an initial
condition near $\tilde{\mathcal{E}}_{0}$, what values must $p$ and
$\tau$ take to prevent an outbreak, or better yet, to wipe
out the infection altogether?
\medskip

Consider first an initial condition near the equilibirum $(S,I,Q)=(\hat 1, \hat 0, \hat 0)$ as in Sec.~\ref{sec:overview}.
In a model with no isolation, the disease reproductive number is known
to be $\mathcal{R}_{0}:=r$. Theorem~\ref{thm:E0} shows  that our isolation
procedure reduces $\mathcal{R}_{0}$ to the \textit{effective disease reproductive number} $\mathcal{R}_{\varepsilon}:=(1-\varepsilon)r$; this is a direct rephrasing
of the statement that the equilibrium at $(\hat 1, \hat 0, \hat 0)$ is stable if 
$q_c= 1- \frac1r \frac{1}{1-\varepsilon} <0$. Thus starting from near 
$(\hat 1, \hat 0, \hat 0)$, to beat the infection we have to have $(1-\varepsilon)r<1$, equivalently 
$\varepsilon>1-\frac{1}{r}$. 

We now decipher what this means for $p$ and $\tau$. As $\varepsilon = p e^{-\tau}$,
$\varepsilon>1-\frac{1}{r}$ imposes immediately a lower bound
on the isolation probability $p$, namely we must have
\begin{equation} \label{eq:critidentification}
p>p_{c}:=1-\frac{1}{r}.
\end{equation}
Having the capability to identify and properly isolate infectious hosts
alone, however, is insufficient. Response time is of the essence: for each $p>p_c$,
there is a \textit{critical identification time}
\begin{equation}
\tau_c(p) =\ln \frac{p}{p_c}\label{eq:critdiagnosis}
\end{equation}
such that if $\tau>\tau_{c}$, the infection spreads for most initial
conditions, whereas $\tau<\tau_{c}$ guarantees that the infection
will abate. If $p=p_c$, then clearly $\tau_c=0$; this implies that isolation has to be immediate upon infection. The farther $p$ is from $p_c$, the larger $\tau_c(p)$, so that there is a trade-off between probability of isolation and the delay in its implementation. 
\medskip

Consider next an initial condition near the equilibrium $(\widehat{1-q}, \hat 0, \hat q)$
for some fixed $q>0$. Theorem 3 together with an argument analogous to that above shows that in this case, the effective disease reproductive number 
is $\mathcal{R}_{\varepsilon,q}:=(1-q)(1-\varepsilon)r$. 
Then $\mathcal{R}_{\varepsilon,q}<1$ is equivalent to $\varepsilon>1-\frac{1}{r(1-q)}$. From this, we 
deduce the corresponding critical isolation probability $p_c(q)$ and 
critical identification time $\tau_c(p,q)$ for each $p$ as before, as in Corollary~\ref{cor:tauc}. 

An alternate way to understand the effective disease reproductive number 
$\mathcal{R}_{\varepsilon,q}$ for $q>0$ is as follows: For initial conditions near $(\widehat{1-q}, \hat 0, \hat q)$, a fraction $q$ of the population will never leave isolation, and
therefore will not participate in the dynamics. Removing this part of the population 
from the system changes nothing other than that we will have $S+I+Q=1-q$. Now
such a system can be rescaled to one with $\check S + \check I+\check Q=1$,
by setting $\check S = S/(1-q), \check I= I/(1-q)$ and $\check Q = Q/(1-q)$,
but observe from Eqs.~\eqref{eq:S-dyn}--\eqref{eq:Q-dyn} that in this rescaling $r$ is changed as well; it becomes $\check r =(1-q) r$, consistent
with the relation between $\mathcal{R}_\varepsilon$ and 
$\mathcal{R}_{\varepsilon,q}$ above.

\medskip
Finally, we remark that the value of $q_{c}$, which fully dictates the stability
properties of the disease-free equilibria, depends only on $p$ and
$\tau$ \textit{and not on} $\kappa$. That is to say, response capabilities
matter, but isolation duration does not, with regard to the prevention
of outbreaks.

\section{Away from Disease-free Equilibria\label{sec:Away-from-Disease-free}}

We now move away from $\mathcal{E}_{0}$, the set of disease-free
equilibria, to explore dynamics on a more global scale. The condition $p>0$ is assumed throughout.

\subsection{An integral of motion\label{subsec:integral-of-motion}}

It has been pointed out that by construction epidemiological models
including delayed terms oftentimes satisfy some secondary invariant
integral condition \cite{Busenberg1980}. It turns out that in addition
to mass conservation, Eqs.~\eqref{eq:S-dyn}\textendash \eqref{eq:Q-dyn}
possess a second conserved quantity. Let $r$ and $\kappa$
be fixed. We define $H=H^{r,\kappa}:\mathcal{C}\to\mathbb{R}$
by 
\begin{equation}
H\left(\phi\right):=1-\phi_{S}(0)-\phi_{I}(-\kappa) +\negmedspace \int\limits _{-\kappa}^{0}(1-r\phi_{S}(s))\phi_{I}(s)\mbox{d}s\label{eq:H}
\end{equation}
where $\phi=(\phi_{S},\phi_{I},\phi_{Q})$.
\begin{prop}
\label{prop:foliation-sigma}For each fixed $r,\kappa$,
\[
\frac{d}{dt}H(x_{t}(\phi))=0\qquad\mbox{for all }\phi\in\mathcal{C}\mbox{ and }t\ge0,
\]
and the level sets of $H$ define a smooth foliation on $\mathcal{C}$. 
\end{prop}

\begin{proof}
Writing $x(t;\phi)=(S(t),I(t),Q(t))$ for $t\ge0$, we have 
\[
\frac{d}{dt}H(x_{t}(\phi))=-\dot{S}(t)-\dot{I}(t-\kappa)+(1-rS(t)I(t)-(1-rS(t-\kappa))I(t-\kappa) ,
\]
which one checks is equal to $0$ by plugging into Eqs.~\eqref{eq:S-dyn}\textendash \eqref{eq:Q-dyn}.
To show that the level sets of $H$ are codimension 1 submanifolds,
it suffices to check, by the Implicit Function Theorem, that $D_\phi H$, the derivative of $H$, is surjective at each $\phi$. This is true, as for any $\phi$ there exits a $\psi=(\psi_S,\psi_I,\psi_Q)$ such that 
\[
(D_\phi H) \psi = -\psi_S(0) -\psi_I(-\kappa) +r\negmedspace \int\limits _{-\kappa}^{0}(1-r\phi_{S}(s))\psi_{I}(s)\mbox{d}s -r\negmedspace \int\limits _{-\kappa}^{0}\phi_{I}(s)\psi_{S}(s)\mbox{d}s\neq 0,
\]
For example, choose $\psi_I=\hat{0}$, $\psi_S(\theta)=0$ for all $\theta\in[-\tau-\kappa,-\delta]$ and $\psi_S(\theta)=1-\theta/\delta$ for $\theta\in[-\delta,0]$, where $0<r\delta<1/ \sup \phi_I$. 
\end{proof}

For fixed $r$ and $ \kappa$, we let $\mathcal{F}=\mathcal{F}^{r,\kappa}$
denote the foliation given by Proposition \ref{prop:foliation},
and let $\mathcal{F}_{q}:=H^{-1}(q)$. Then $\mathcal{F}_{q}$
is invariant under the semi-flow, i.e., for $\phi\in\mathcal{F}_{q}$,
$x_{t}(\phi)\in\mathcal{F}_{q}$ for all $t\ge0$. We consider below
the intersection of $\mathcal{F}_q$ with the set of equilibrium points
for arbitrary $\kappa$ and $q$.

\begin{thm}
\label{thm:candidate}
For fixed $r, \kappa, q$, we let $\mathcal{F}=\mathcal{F}^{r,\kappa}$, 
and consider $\mathcal{F}_q$.
\begin{enumerate}
\item[(1)] Then $\mathcal{F}_{q}\cap\mathcal{E}_{0}=\{\widehat{u(q)}\}$, where $u(q)=(1-q,0,q)$.
\item[(2)] Fixing additionally $p,\tau$, which determines
  $\mathcal{E}_{I}=\{\phi_S =
[r(1-\varepsilon)]^{-1}\}$, we have $\mathcal{F}_{q}\cap\mathcal{E}_{I}=\{\widehat{v(q)}\}$, where $v(q)=(v_S,v_I(q),v_Q(q))$ and
\[
v_{I}(q)=\frac{1-\varepsilon}{1-\varepsilon+\varepsilon\kappa}\left(q_{c}-q\right),
\] 
\begin{equation}
v_{Q}(q)=\frac{\varepsilon\kappa}{1-\varepsilon+\varepsilon\kappa}\left(q_{c}-q\right).
\label{eq:def-h}
\end{equation}
\end{enumerate}
\end{thm}
These assertions follow from straightforward computations. 

We remark on how the leaves of $\mathcal{F}^{r,\kappa}$ 
vary with $\kappa$.
Setting $\kappa=0$, we see from (\ref{eq:H}) that $\mathcal{F}_{q}=\{\phi_{Q}(0)=q\}$.
For small $\kappa>0$, it is easy to see that the leaves $\mathcal{F}_{q}^{r,\kappa}$
are ``close'' to those of $\mathcal{F}_{q}^{r,0}$. 
Observe from the formulas above that with $p$ and $\tau$ fixed, 
$v_I(q)$ decreases monotonically as $\kappa$ increases.
Indeed the leaf $\mathcal{F}^{r,\kappa}_{q}$ ``bends'' away from 
$\mathcal{F}^{r,0}_{q}$ increasingly, its intersection with $\mathcal{E}_{I}$ tending 
to $(\widehat{1-q_{c}},\hat{0},\hat{q}_{c})$ as $\kappa\to\infty$.


\subsection{Discussion\label{subsec:Discussion}}

Fixing $r,\kappa$ and starting from $\phi$ with $\phi_{I}\ll1$,
one asks what the future holds. For fixed $p,\tau$,  suppose $\phi\in\mathcal{F}_{q}$
for some $q$. If $q>q_{c}$, then $x_{t}(\phi)\to \widehat{u(q)}\in\mathcal{F}_{q}\cap\mathcal{E}_{0}$
as $t\to\infty$ by Theorem \ref{thm:E0}, if $\phi$ was chosen in
some sufficiently small neighborhood of $\mathcal{E}_{0}$. We focus
therefore on the case $q<q_{c}$, for which we have to expect
$x(t;\phi)$ to move away from the set of disease-free equilibria
$\mathcal{E}_{0}$.

One possibility is for $x_{t}(\phi)$ to tend to $\widehat{v(q)}$,
the unique point in $\mathcal{F}_{q}\cap\mathcal{E}_{I}$, as $t\to\infty$.
It is difficult to determine if, or under what conditions, this occurs;
such nonlocal dynamical behaviors are very challenging to analyze.
We have some evidence that this is not an unreasonable expectation,
at least for smaller values of $\kappa$, and confirmed this with
 numerical simulations; see Figure 6.1.

Not all endemic equilibria identified in Sec.~\ref{subsec:Equilibrium}
are reachable if one starts from an initial condition $\phi$ near
$\tilde{\mathcal{E}}_{0}$. For each $r,p,\tau,\kappa$, we define the
set of \textit{reachable endemic equilibria} $\mathcal{E}_{I}^{r,p,\tau,\kappa}$
to be those equilibrium points in $\mathcal E_I$ that are, in principle, 
reachable starting from a biologically realistic initial condition, i.e.,
\begin{eqnarray*}
\mathcal{E}_{I}^{r,p,\tau,\kappa} & = & \{\hat{v}, v=(v_S,v_I(q),v_Q(q)), 
q \in [0,1], q <q_c\}\ ,
\end{eqnarray*}
where $v_S,v_I(q)$ and $,v_Q(q)$ are as in Theorem \ref{thm:candidate}.

In the scenario that $x_{t}(\phi)$ tends to $\hat{v}$, Theorem
\ref{thm:candidate} tells us it is advantageous to use a larger $\kappa$,
for the longer one keeps infectious hosts in isolation, the smaller
the $I$-component $v_I$ of the asymptotic state $\hat{v}$. If
$\hat{v}$ is unstable, then convergence to it
is unlikely, and the structures that emerge from $\hat{v}$
after it loses stability become candidates for the $\omega$-limit
set of $\phi$, which we know is nonempty if $x(t;\phi)$
is bounded (by the remark at the end of Sec. \ref{subsec:Equilibrium}). 
This motivates the eigenvalue
analysis of the equilibria in $\mathcal{E}_{I}$ in the next section.

\begin{figure}
\noindent\begin{minipage}[t]{1\columnwidth}%
\includegraphics[width=1\linewidth]{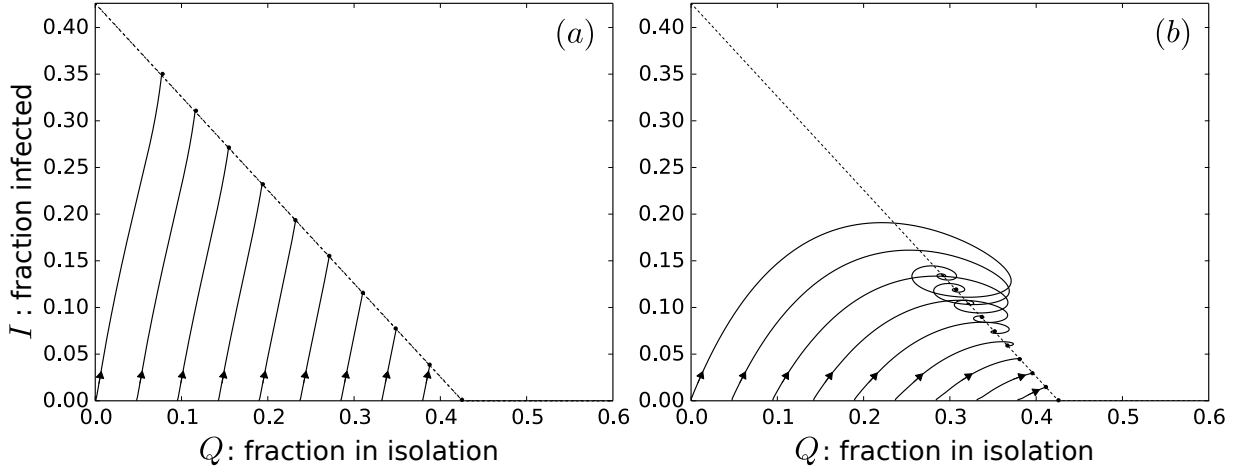}%
\end{minipage}\caption{\label{fig:reach}Convergence to $\mathcal{\tilde{E}}_{I}$ for initial
conditions close to $\tilde{\mathcal{E}}_{0}$. Here $p=\frac{1}{2},\tau=0.5$.
Initial conditions are of the form $\phi(\theta)=(1,0,0)$ for all
$\theta<0$ and $\phi(0)=(0.999-q,0.001,q)$ for 10 equidistant
values of $0\leq q\leq q_{c}$. Panel (a): $\kappa=0.5$. Panel
(b): $\kappa=5$. Curves show trajectories in direction of arrows
projected onto $(Q,I)$-space. }
\end{figure}

\section{The case of an endemic infection \label{sec:Case-of-endemic-infection} }

In Secs. \ref{subsec:endemic-Lin} and \ref{subsec:endemic-nonlin},
we study the dynamics close to $\mathcal{E}_{I}$, the set of endemic
equilibria defined in Sec. \ref{subsec:Equilibrium}. For fixed $r>1,p\in(0,1)$
and $\tau=0$, we give in Sec. \ref{subsec:endemic-Lin} a complete
bifurcation analysis of each equilibrium point in $\mathcal{E}_{I}$
as $\kappa$ increases. These results remain valid for small $\tau>0$.
In Sec.~\ref{subsec:endemic-nonlin}, we deduce from the linear analysis
above nonlinear behaviors in neighborhoods of these equilibria. 

While Secs. \ref{subsec:endemic-Lin} and \ref{subsec:endemic-nonlin}
are concerned with the dynamical picture near an endemic equilibrium irrespective of how one gets there, Sec. \ref{subsec:endemic-relev} addresses the following very pertinent question: Given an initial condition $\phi$ with small $\phi_{I}>0$, 
if one is unable to control the outbreak, which $\kappa$, i.e.,
what durations of isolation, will best mitigate the severity of the infection?
 As we will show, the dynamical landscape
 is quite complex. Results of numerical computations will be presented 
 to clarify the situation.


\subsection{Linear analysis at $\mathcal{E}_{I}$\label{subsec:endemic-Lin}}

Let $r$ and $\varepsilon$ be fixed throughout. We parametrize $\mathcal{E}_{I}$
by $\widehat{w(q)},\,q\in\mathbb{R},$ where $w\!\left(q\right)=\left(1-q_{c},q_{c}-q,q\right)$ 
and study the linearized equation at each point. Clearly, $0$ is
an eigenvalue, as $\mathcal{E}_{I}$ is a line of equilibria. We have
the following result for $\kappa=0$. 
\begin{prop}
\label{prop:L1-lin} Let $\tau$ and $p\in(0,1)$ be fixed, and $\kappa=0$.
If $q\leq q_{c}$ then $\widehat{w(q)}$ is linearly
stable; otherwise it is linearly unstable.

Specifically, for $q<q_{c}$, the eigenvalue $\lambda=0$ has
multiplicity 1 and all other eigenvalues satisfy $Re\left(\lambda\right)<0$,
and one eigenvalue crosses the imaginary axis as $q$ increases
past $q_{c}$. 
\end{prop}
The proof of Proposition~\ref{prop:L1-lin} is analogous to the proof
of the stability of disease-free equilibria in Theorem~\ref{thm:E0}.
More specifically, for the case $\kappa=0$, the corresponding characteristic
equation is 
\begin{equation}
\lambda\left(\lambda+1-r\left(1-q-2(q_{c}-q)\right)\left(1-\varepsilon e^{-\tau\lambda}\right)\right)=0,\label{eq:d1kappa0-spec}
\end{equation}
which has the same form as Eq.~(\ref{eq:I=00003D00003D0-hyperbolic_spec})
from Theorem~\ref{thm:E0}. Therefore, the statement of Proposition~\ref{prop:L1-lin}
can be proven by similar arguments. 

We remark that the stability persists at least for small values of
$\kappa$ for all points $\widehat{w(q)}$, $q \ne q_c$. Moreover, a uniform
estimate for such $\kappa$ can be obtained by excluding a neighborhood
of the point $\widehat{w(q_c)}$. 

For $q<q_{c}$, even as Proposition \ref{prop:L1-lin} tells us
that $\widehat{w(q)}$ is stable for small $\kappa$, there
is no guarantee that it will not destabilize for larger values of
$\kappa$. We first give a rigorous analysis for the case $\tau=0$,
fixing $q$ and letting $\kappa$ increase, as there are
standard techniques for investigating asymptotic properties of the
spectrum as the delay increases. We refer to \cite{Lichtner2011}
for a general overview of the concepts used in the proof of the following
theorem. 
\begin{thm}
\label{thm:L1-lin-Hopf} Let $0<p<1$ and $\tau=0$ be fixed. Then,
there exist $q_{h}^{-},q_{h}^{+}$ for which the following hold:
If $q\in\left[q_{h}^{-},q_{h}^{+}\right]$ and $q<q_{c}$,
then $\widehat{w(q)}$ is linearly stable for all $\kappa\geq0$.
For $q$ such that $q\notin\left[q_{h}^{-},q_{h}^{+}\right]$
and $q<q_{c}$, we have the following.

\begin{enumerate}
\item There exists $\kappa_{0}(q)$ such that $\widehat{w(q)}$ is linearly
stable for $\kappa\leq\kappa_{0}(q)$ and linearly unstable for
$\kappa>\kappa_{0}(q)$. 
\item For $\kappa=\kappa_{0}(q)$, the linearization at $\widehat{w(q)}$
possesses a pair of purely imaginary eigenvalues $\pm i\Omega(q)$,
$\Omega(q)>0$, crossing the imaginary axis with positive speed
as $\kappa$ increases. 
\item For each $\kappa_{m}(q)=\kappa_{0}(q)+m2\pi/\Omega(q)$, $m\in\mathbb{N}$,
$\widehat{w(q)}$ possesses a pair of purely imaginary eigenvalues $i\Omega(q)$,
crossing the imaginary axis with positive speed. 
\end{enumerate}
For $q<q_{c}$, these are the only bifurcations as $\kappa$ is
varied. 
\end{thm}

The results of Theorem 9 carry the following biological interpretation: Suppose
we find ourselves near an endemic equilibrium. How we got there is of no concern
--  be it due to natural calamity, large stochastic fluctuations, viral mutation -- what
matters is that we are there, and the question is: what are the effects of 
prolonged periods of isolation? Theorem 9 gives a complete answer to this
question on the linear level.


\begin{proof}[Proof of Theorem \ref{thm:L1-lin-Hopf}]
We compute the eigenvalues of the linearization at $\widehat{w(q)}$.
By Lemma \ref{lem:characteristicf}, the characteristic equation at
$\widehat{w(q)}$ reads
\begin{align}
\chi\left(\lambda,\widehat{w(q)},\kappa\right) = & \lambda\left(\lambda+1-r\left(1-q_{c}\right)\left(1-\varepsilon e^{-\tau\lambda}\right)+r\left(q_{c}-q\right)\left(1-\varepsilon e^{-\tau\lambda}Y(\lambda)\right)\right) \label{eq:spec-end} \\
 & +r\left(q_{c}-q\right)\varepsilon e^{-\tau\lambda}\left(1-Y(\lambda)\right), \nonumber
\end{align}
where we use the notation $e^{-\lambda\kappa}=:Y(\lambda)$. Note that the solution $\lambda_{0}=0$ of \eqref{eq:spec-end} corresponds to the direction
tangential to the line $\mathcal{E}_{I}$. We use the result from
\cite{Lichtner2011}, which describes the asymptotic properties of
the spectrum for large delay (here $\kappa$). More specifically,
the spectrum for large $\kappa$ can be described by two parts: the
strong spectrum $\lambda$ such that $\lambda=\mathcal{O}\left(1\right)$
as $\kappa\to\infty$, which is given by the solutions of the equation
\begin{equation}
\lambda\left(\lambda+1-r\left(1-q_{c}\right)\left(1-\Lambda\left(\lambda\right)\right)+r\left(q_{c}-q\right)\right)+r\left(q_{c}-q\right)=0,\label{eq:spec-strong}
\end{equation}
(Eq.~\eqref{eq:spec-end} for $Y=0$) with positive real parts, and
the so called asymptotic continuous spectrum with $\mbox{Re}\,\lambda=\mathcal{O}\left(1/\kappa\right)$
as $\kappa\to\infty$. For $\tau=0$, Eq.~\eqref{eq:spec-strong} has
the solutions 
\[
\lambda_{\pm}=\frac{1}{2}\left[-r\left(1-p\right)\left(q_{c}-q\right)\pm\sqrt{\left(q_{c}-q\right)\left(r^{2}\left(1-p\right)^{2}\left(q_{c}-q\right)-4p\right)}\right],
\]
satisfying $Re(\lambda)<0$ for all $q\in\left[0,q_{c}\right)$,
which means that the strong spectrum is absent. The asymptotic continuous
spectrum is given by $\lambda=-\frac{1}{\kappa}\gamma\left(\omega\right)+i\omega$
where $\gamma\left(\omega\right)=-\frac{1}{2}\log\left|Y(i\omega)\right|^{2}$,
and $Y(\cdot)$ can be computed by solving \eqref{eq:spec-end} with
respect to $Y$ (see more details in \cite{Lichtner2011,YanchukGiacomelli2017})
\begin{eqnarray*}
Y(\lambda) & = & \frac{\lambda^{2}+\lambda\left(1+r\left(1-q\right)+\frac{1}{1-\varepsilon}\Lambda\left(\lambda\right)\right)+r\left(q_{c}-q\right)\Lambda\left(\lambda\right)}{r\left(q_{c}-q\right)\Lambda\left(\lambda\right)\left(\lambda+1\right)}.
\end{eqnarray*}
It is straightforward to compute that $\gamma\left(0\right)=0.$ Moreover,
$\frac{\partial\left|Y\right|}{\partial\omega}(0)=0$ and $\frac{\partial^{2}\left|Y\right|}{\partial\omega^{2}}(0)=h(q)$
with 
\[
h\left(q\right)=\left(\frac{1-r\left(1-p+\left(p-2\right)q_{c}-q\right)}{pr\left(q_{c}-q\right)}\right)^{2}-\frac{2}{pr\left(q_{c}-q\right)}-1.
\]
Hence, $\frac{\partial^{2}Y}{\partial\omega^{2}}(0)$ changes sign
when $h(q)=0$. In particular, $h(q)<0$ corresponds to the so-called
modulational instability \cite{YanchukGiacomelli2017}. Simple analysis
of the function $h(q)$ shows that $h\left(q\right)\geq0$ for
all $q\in\left[q_{h}^{-},q_{h}^{+}\right]\neq\emptyset,$ where
\begin{equation}
\left(1-p^{2}\right)\left(q_{c}-q_{h}^{\pm}\right)=a+p\mp\sqrt{\left(a+p\right)^{2}-\left(1-p^{2}\right)a^{2}}\label{eq:Hopfbounds}
\end{equation}
and $a=p_c -p+\left(p-3\right)q_{c}$. In this case,
$\gamma\left(\omega\right)$ is concave and $\gamma\left(\omega\right)<0$
for all $\omega\in\mathbb{R}\setminus\left\{ 0\right\} .$ As a result,
there are no eigenvalues with positive real part for sufficiently
large $\kappa$. In contrast, for all $q\notin\left[q_{h}^{-},q_{h}^{+}\right]$
there exists an open set $I_{\Omega}:=\left(-\Omega,\Omega\right)\setminus\left\{ 0\right\} $
such that $\gamma(\pm\Omega)=0$ and the pseudo-continuous spectrum
$\gamma(\omega)>0$ for $\omega\in I_{\Omega}\setminus\left\{ 0\right\} $.
Hence, as follows from \cite{Lichtner2011} for large $\kappa$ there
exists at least one pair of complex conjugated eigenvalues with positive
real parts and nonzero imaginary parts. With the increasing of $\kappa$,
these eigenvalues have to cross the imaginary axis at $\pm i\Omega$.
We denote the corresponding value of $\kappa$, where this occurs
as $\kappa_{0}(q)$.

Hence, it holds that $\chi\left(\pm i\Omega(q),\widehat{w(q)},\kappa_{0}(q)\right)=0$
for some $\kappa_{0}(q)>0$. Then, for $\kappa_{m}(q)=\kappa_{0}(q)+m2\pi/\Omega(q)$,
$m\in\mathbb{N}$ it holds 
\[
\chi\left(\pm i\Omega(q),\widehat{w(q)},\kappa_{m}(q)\right)=\chi\left(\pm i\Omega(q),\widehat{w(q)},\kappa_{0}(q)\right)=0,
\]
since $\left.Y\left(\pm i\Omega\right)\right|_{\kappa_{m}}=e^{\mp i\Omega\kappa_{m}}=e^{\mp i\Omega\kappa_{0}}=\left.Y\left(\pm i\Omega\right)\right|_{\kappa_{0}}.$
This implies the existence of purely imaginary eigenvalues at all
values $\kappa_{m}(q)$, which form the diverging monotone sequence
of delay values for each point $\widehat{w(q)}$. 
\end{proof}

\begin{figure}[h]
\noindent %
\noindent\begin{minipage}[t]{1\columnwidth}%
\centering
\includegraphics[width=\linewidth]{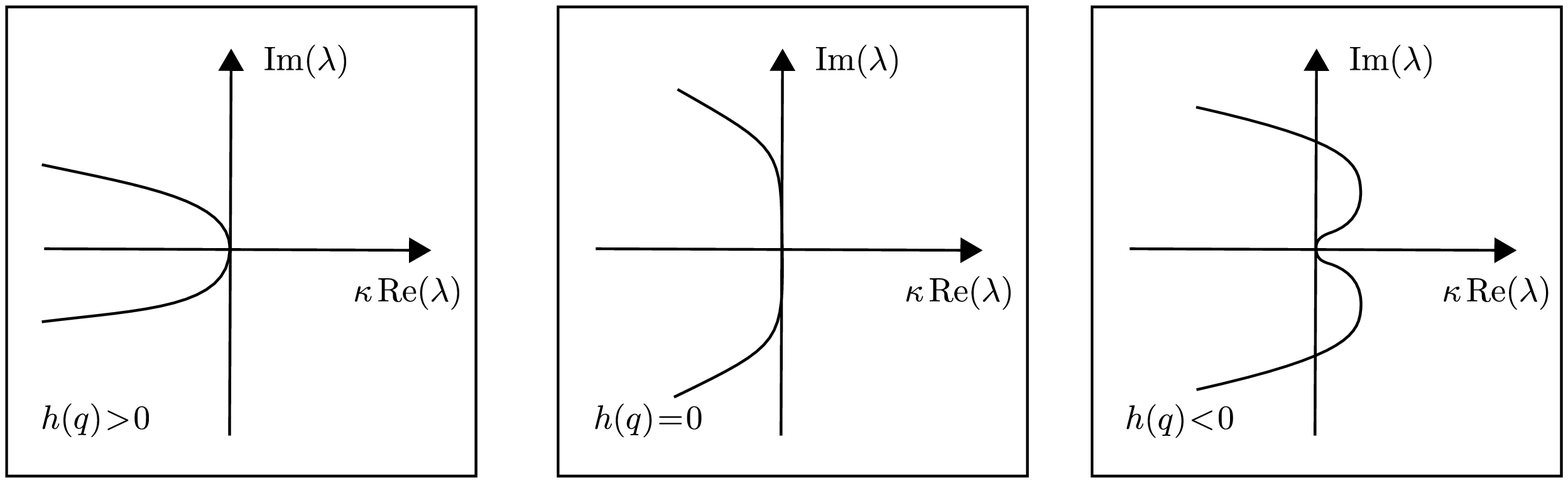}%
\end{minipage}\caption{Asymptotic spectral properties of a given equlibrium 
$\widehat{w(q)} \in\tilde{\mathcal{E}}_{I}$
for large values of $\kappa$. All eigenvalues close to the imaginary
axis lie on invariant curves with shapes indicated as above. In particular,
$\lambda=0$ is an eigenvalue and the distance between two neighboring
eigenvalues on these curves is approximately $2\pi/\kappa$. The shapes
can be distinguished by the sign of the auxilliary variable $h(q)$
introduced in the proof of Theorem~\ref{thm:L1-lin-Hopf}, where
the case $h(q)<0$ corresponds to the so called modulational instability,
see Ref.~\cite{Lichtner2011}. \label{fig:spec}}
\end{figure}

By standard theory, eigenvalues at $\widehat{w(q)}$ for small $\tau>0$
are close to those at $\tau=0$. Thus for each $\kappa$ away from
bifurcation points, $\widehat{w(q)}$ has the same number of unstable
eigenvalues for small $\tau$ as in the theorem above, with the size
of the allowed perturbation in $\tau$ depending on $\kappa$. See
Fig \ref{fig:spec} for graphic visualization of the spectral properties
described above. We have also computed numerically 
the regions of stability for a range of values of $q$ and $\kappa$; they are shown
 in Fig.~\ref{fig:stabilitykappa}.

\begin{figure}
\noindent %
\noindent\begin{minipage}[t]{1\columnwidth}%
\includegraphics[width=1\linewidth]{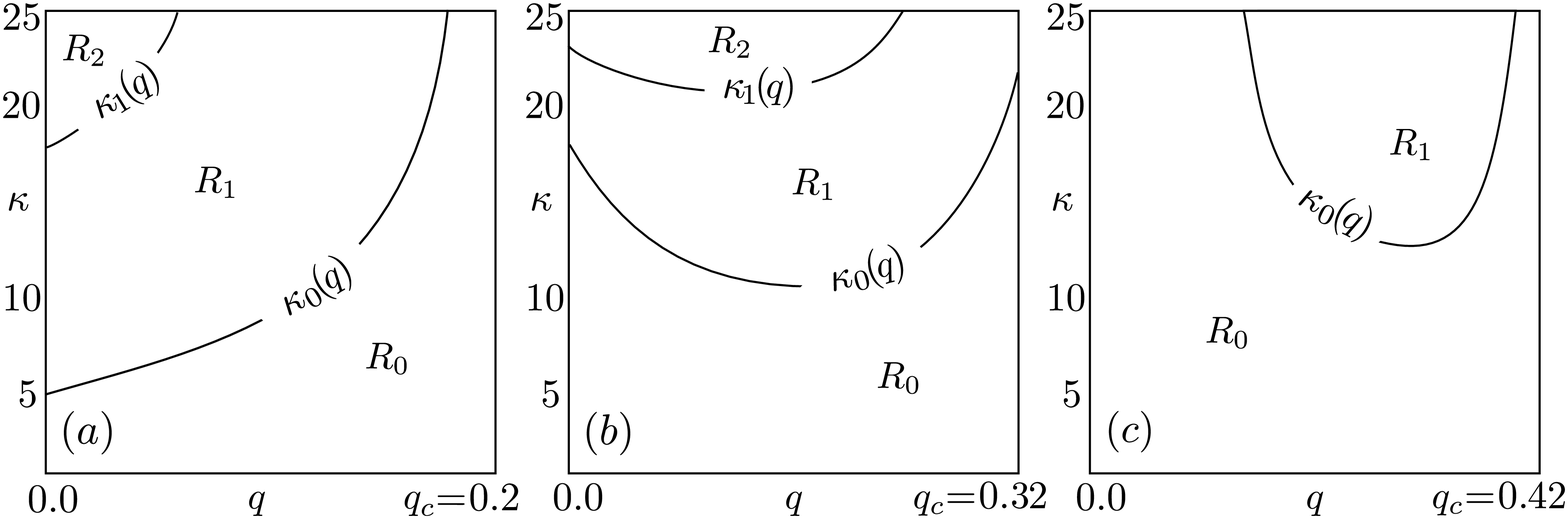}%
\end{minipage}\caption{\label{fig:stabilitykappa}Stability of the endemic equilibria $\widehat{w(q)}\in\tilde{\mathcal{E}}_{I}$
as function of $q$ and $\kappa$. Here $r=2.5$, $p=0.5$. Panel
(a): $\tau=0$, Panel (b): $\tau=0.2$, Panel (c): $\tau=0.5$. Each
panel shows $q\in[0,q_{c}]$ in the $x$-axis, and $\kappa\in[0,25]$
in the $y$-axis. The square is partitioned into $R_{0}=\{(q,\kappa):\kappa<\kappa_{0}(q)\}$,
$R_{i}=\{(q,\kappa):\kappa_{i-1}(q)<\kappa<\kappa_{i}(q)\},\ i=1,2$.
Drawing a vertical line through the square, one witnesses the implications
of Theorem \ref{thm:L1-lin-Hopf}; that is, for fixed $0\leq q<q_{c}$,
the equilibrium $\widehat{w(q)}$ is increasingly destabilized as $\kappa$
increases. Drawing a horizontal line, one sees at a glance the regions
of stability and instability for each $\kappa$. }
\end{figure}



\subsection{The nonlinear picture near $\mathcal{E}_{I}$\label{subsec:endemic-nonlin}}

As the semi-flow $T^{t}$ is $C^{1}$ (see Sec. 3.1), we have at
our disposal stable and unstable manifolds theory to further clarify
the nonlinear picture near $\mathcal{E}_{I}$ as was done for $\mathcal{E}_{0}$.
Additionally, we know from Sec. \ref{subsec:integral-of-motion}
that for each $r, \kappa, p,\tau$, there is a $T^{t}$-invariant, codimension
1 foliation $\mathcal{F}=\mathcal{F}^{r,\kappa}$ transversal
to $\mathcal{E}_{I}$. Below we let $\mathcal{F}_{q'}$ be the 
leaf of $\mathcal{F}$ passing through $\widehat{w(q)}$, so that
$q'$ and $q$ are related by $q=v_Q(q')$ where $v_Q$ is as in Theorem 7.

Consider $q<q_{c}$. By Proposition \ref{prop:L1-lin}, for $\kappa<\kappa_{0}$,
$\widehat{w(q)}$ is an attractive fixed point for the dynamics on $\mathcal{F}_{q'}$,
so that any orbit on $\mathcal{F}_{q'}$ coming to within a certain
distance of $\widehat{w(q)}$ (measured along $\mathcal{F}_{q'}$)
will converge to it. For $\kappa<\kappa_{0}$ and close enough
$\kappa_{0}$, we know from the complex conjugate eigenvalues at $\widehat{w(q)}$
that any such trajectory will exhibit damped oscillatory behavior
as it tends to its endemic equilibrium. Though not necessarily the
case, this will likely be reflected also in $I(t)$, the fraction
of population infectious. Though Theorem 9 cannot
be applied directly to the situation depicted in Fig 6.1(b), the presence 
of complex eigenvalues  is consistent with the way some of the trajectories 
spiral toward their endemic equilibria.

At $\kappa=\kappa_{0}\left(q\right)$, a Hopf bifurcation occurs
at $\widehat{w(q)}$. Though technical conditions are difficult to check,
in a generic super-critical Hopf bifurcation what happens is that
for $\kappa$ just past $\kappa_{0}$ a small limit cycle emerges
from $\widehat{w(q)}$. More precisely, restricted to $\mathcal{F}_{q'}$,
the dynamics near $\widehat{w(q)}$ can be described as follows: There
is a strong stable manifold, codimension 2 with respect to $\mathcal{F}_{q'}$,
and a 2D center manifold passing through $\widehat{w(q)}$. All orbits
on $\mathcal{F}_{q'}$ that are within a certain distance of $\widehat{w(q)}$
are driven towards the 2D center manifold, towards the small limit
cycle bifurcating from $\widehat{w(q)}$. This dynamical picture persists
as $\kappa$ increases, at least for a little while; the limit cycle
grows larger and becomes more robust.

By the time $\kappa$ reaches $\kappa_{1}$, it is difficult to know
if the picture above still persists. If it does, then what happens
as $\kappa$ increases past $\kappa_{1}$ is that the codimension
2 strong stable manifold within $\mathcal{F}_{q'}$ becomes codimension
4, and orbits on $\mathcal{F}_{q'}$ near $\widehat{w(q)}$ are driven
towards a 4D center manifold. A second frequency of oscillation with
small amplitude develops around the existing larger and more robust
limit cycle. At each $\kappa_{m}$, the dimension of the center manifold
goes up by 2.


\subsection{Optimizing isolation durations\label{subsec:endemic-relev}}

In the last two subsections, we have focused on the dynamical properties
near specific equilibria in $\mathcal{E}_{I}$ for specific parameters.
That information is useful, but the question of practical importance
here is the following. Suppose we find ourselves at some $\phi=(\phi_{S},\phi_{I},\phi_{Q})\in\tilde{\mathcal{C}}$
with $0<\phi_{I}\ll1$, and the response capabilities, i.e. $p$ and
$\tau$, are such that they are not sufficient for preventing an outbreak
given the reproductive number $r$ of the disease. That is to say,
$\phi_{Q}(0)<q_{c}$. Accepting that the infection will become endemic,
the question is: will some lengths of isolation be more effective
in mitigating the outbreak, and are there optimal choices of $\kappa$?
Assuming $r,p$ and $\tau$ are fixed, we propose the following two
sets of considerations:

\medskip \noindent
\textbf{The potential endemic equilibrium}. First, there is the endemic
equilibrium to which $x(t;\phi)$ may \textendash{} or may not \textendash{}
eventually tend. This can be computed as follows: For each $\kappa$,
we compute $q':=H^{r,\kappa}(\phi)$ where $H$ is as in 
Sec.~\ref{subsec:integral-of-motion}.
This determines $\mathcal{F}_{q'}$, the leaf of the foliation $\mathcal{F}^{r,\kappa}$
containing the initial condition $\phi$. From (\ref{eq:def-h}),
we compute explicitly $I(\phi,\kappa)=v_I(q')$, the $I$-coordinate of the 
the point in $\mathcal{E}_{I}$ to which the trajectory
from $\phi$ may potentially be attracted if the duration of isolation
is $\kappa$.

Fig.~\ref{fig:sameinitial}(a) shows the trajectories for a few initial
conditions with $0<\phi_{I}\ll1$ close to the point $(\hat 1, \hat 0,\hat 0)$ in
$\tilde{\mathcal{E}}_{0}$. Here we see that for $\kappa$ up to about
10, the solution converges to a stable equilibrium, with what appears
to be a Hopf bifurcation occurring around $\kappa=10$. This is related
to, though not strictly the same as, the Hopf bifurcation in Theorem~\ref{thm:L1-lin-Hopf}:
here as we vary $\kappa$, the point in $\mathcal{E}_{I}$ changes
with it. Long before this bifurcation, the complex conjugate eigenvalues
of the points in $\mathcal{E}_{I}$ (see Theorem~\ref{thm:L1-lin-Hopf})
are clearly visible, as the solutions spiral toward the equilibria.
This translates into oscillatory behavior for $I(t)$, the fraction
of the population that is infectious. Before the bifurcation, these oscillations
are damped; the damping grows weaker and eventually disappears altogether.
As shown in Fig.~\ref{fig:sameinitial}(b), for larger $\kappa$,
the solutions tend to limit cycles which appear to grow in size, with
$I(t)$ rising periodically higher than some of the stable equilibria
to which solutions tend for smaller $\kappa$.

\begin{figure}
\noindent %
\noindent\begin{minipage}[t]{1\columnwidth}%
\includegraphics[width=1\linewidth]{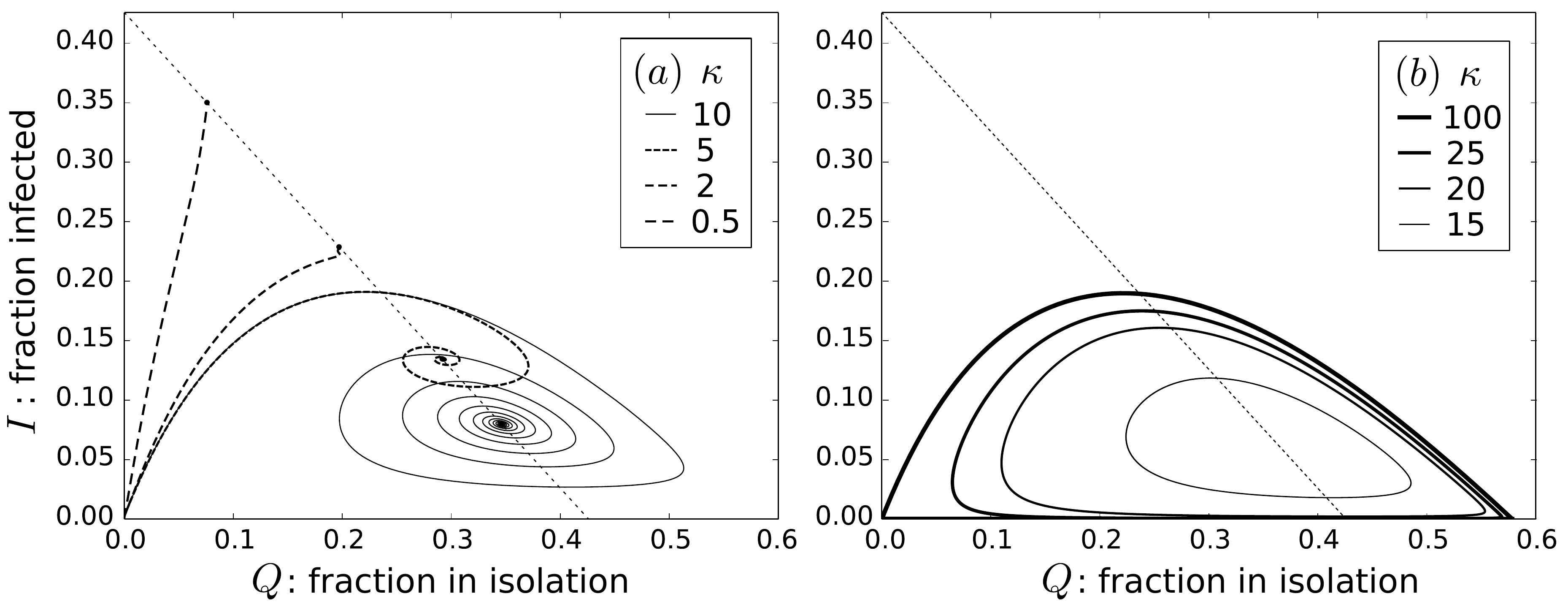}%
\end{minipage}\caption{\label{fig:sameinitial} Long term behavior for solutions with set
initial condition close to $\mathcal{\tilde{E}}_{0}$. Fix $p=\frac{1}{2}$,
$\tau=0.5$ and initial condition $\phi$, such that $\phi(\theta)=(1,0,0)$
for all $\theta\in[-\tau-\kappa,0)$ and $\phi(0)=(0.99,0.01,0)$
(Lemma \ref{lem:positivity} guarantees that $x_{t}(\phi)\in\tilde{\mathcal{C}}$
for all $t\geq\tau+\kappa$). Panel (a): For small values of $\kappa$
the solution $x_{t}(\phi)$ converges to some $\hat{v}\in\mathcal{\tilde{E}}_{I}$
given by Theorem \ref{thm:candidate}. The rate of convergence, however,
decreases as $\kappa$ grows until there is a super-critical Hopf-bifurcation
and the solution $x(t;\phi)$ converges to a limit cycle. Panel (b):
Limit cycles of $x(t;\phi)$ for $\kappa\in\left\{ 15,20,25,100\right\} .$
See also Fig.~\ref{fig:stabilitykappa}(b). }
\end{figure}

We remark that the trajectories depicted in Fig 6.3 are
likely representatives of trajectories starting near $\phi$. This
is because through each $\widehat{u(q)}$ where $u(q)=(1-q,0,q)$
with $q<q_{c}$, there is, within $\mathcal{F}_{q}$, a codimension-1
stable manifold $W^{s}$ and a 1D unstable manifold $W^{u}$. Starting
from any $\phi\in\mathcal{F}_{q}$ with $\phi_{I}\ll1$, assuming
$\phi \not \in W^{s}$, its trajectory will follow $W^{u}$, which
consists of a single trajectory. As this is true for all $\phi\in\mathcal{F}_{q}$
with $\phi_{I}\ll1$, examining where one trajectory goes will tell
us about all such trajectories.

\medskip \noindent
\textbf{The maximum outbreak size.} Above we were concerned with the
large-time dynamics of the disease, the eventual level of infection.
Here we look at the transient dynamics 
before this asymptotic state is reached.
For each $\phi$ and $\kappa$, we define 
\[
I_{{\rm peak}}(\phi,\kappa)=\sup_{t\ge0}\ I(t)
\]
where $I(t)$ is the $I$-component of $x(t;\phi)$.
This is a very relevant quantity, as too large an $I_{{\rm peak}}$-value
is clearly unacceptable even if eventually the disease winds down.

Fig.~\ref{fig:Ipeak} shows this quantity as a function of
$\kappa$ for the same $\phi$ in Fig.~\ref{fig:sameinitial} for a few
values of $\tau$.  These plots show that $I_{{\rm peak}}$ is a decreasing function
which levels off beyond a certain point, i.e., even though  $I(\phi,\kappa)$
continues to decrease 
with increasing
$\kappa$, the worst of the epidemic does not improve. That is to
say, the time course of the infection is such that it will
first get worse, and only after a certain fraction of the population is infectious
that it will start to abate, due to the effect of isolation, which
diminishes the size of the susceptible population. 

\begin{figure}
\begin{minipage}[t]{0.5\columnwidth}%
\includegraphics[width=1\linewidth]{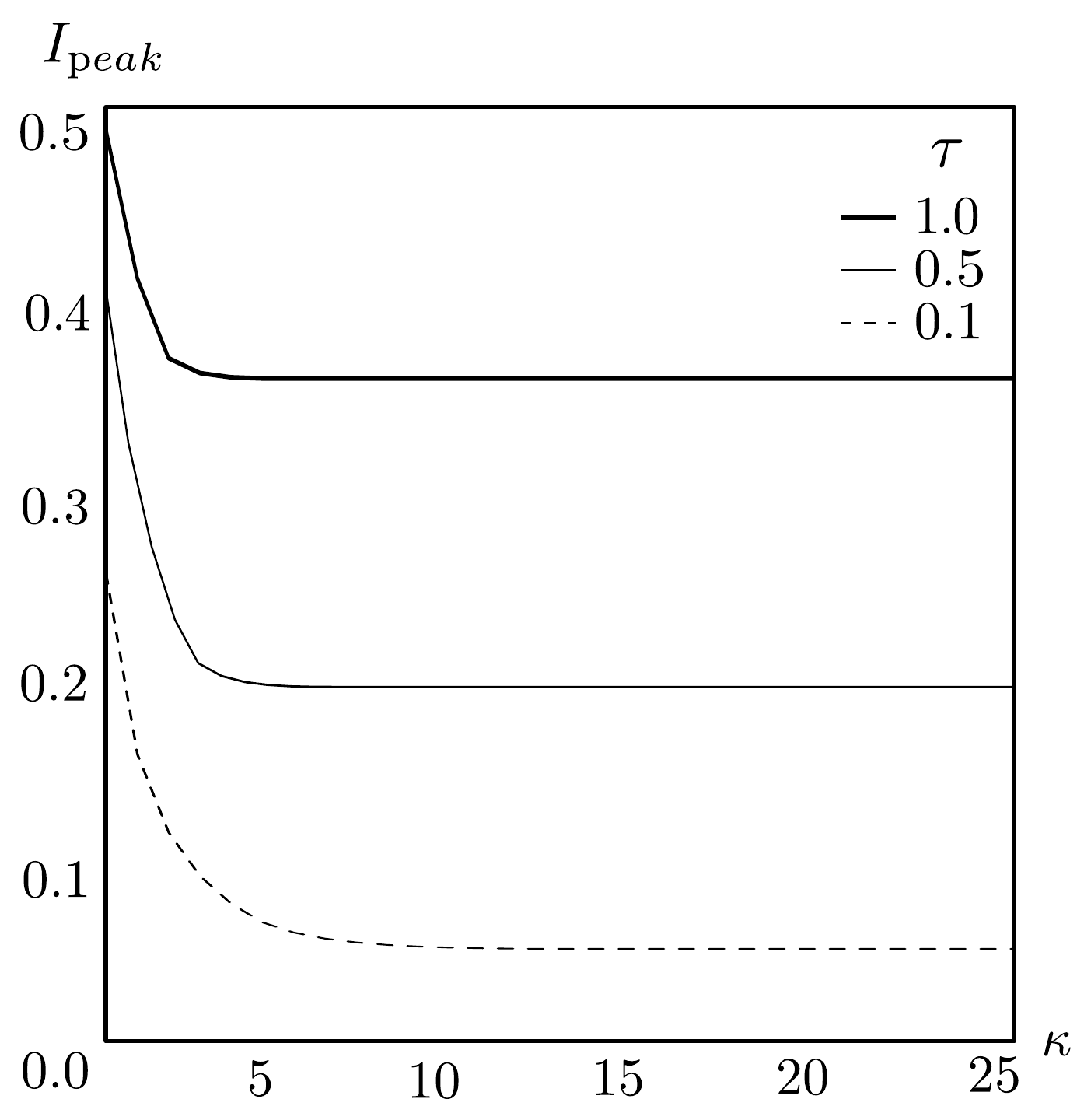}%
\end{minipage}\caption{\label{fig:Ipeak}Values of $I_{{\rm peak}}$ for $r=2.5,p=0.5$ as
functions of $\kappa\in[0,25]$ for fixed $\tau\in\left\{ 0.1,0.5,1\right\} $
and initial condition $\phi(\theta)=(1,0,0)$ for all $\theta<0$
and $\phi(0)=(0.999,0.001,0)$.}
\end{figure}

Here is a rigorous argument for why $I_{{\rm peak}}$ is bounded below
by a positive value independent of $\kappa$: Consider the limiting
case $\kappa=\infty$, i.e., individuals that enter the state $(Q)$
remain there forever. The unstable manifold at the point $(\hat 1,\hat 0,\hat 0)$
is a curve whose $I$-component increases initially and must eventually
tend to $0$ as the entire population is in $(Q)$. Denoting the maximum
value of the $I$-component of this unstable curve by $\tilde{I}_{{\rm peak}}$,
it is easy to see that for $\phi$ with $0<\phi_{I}\ll1$ near $(\hat 1,\hat 0,\hat 0)$
and any $\kappa$, we must have $I_{{\rm peak}}(\phi,\kappa)\ge\tilde{I}_{{\rm peak}}$:
the part of the population that leaves isolation becomes susceptible
and can only contribute to a larger $I(t)$. 

Finally, we discuss the question posed at the beginning of this section:
What constitutes an optimal value of $\kappa$, in the setting above
where $p$ and $\tau$ are fixed and $\phi$ is given? First one has
to decide whether it is the value of $I_{{\rm peak}}$ that matters,
or the eventual level of infection. With regard to large-time dynamics, 
there is also the following consideration: If the trajectory tends to an
equilibrium $\hat{w}\in\mathcal{E}_{I}$, then obviously the
smaller the $I$-component of $\hat{w}$, the better.
As noted in Theroem \ref{thm:candidate}, this means taking as large
 $\kappa$ as we can. But too large a value of $\kappa$ is also
impractical. Also, for larger values of $\kappa$, 
$\hat{w}$ can destabilize, with the trajectory
accumulating on a limit cycle, as shown in Fig.~\ref{fig:sameinitial}.
This means $I(t)$ will oscillate forever
periodically in time, with potentially higher peaks (as well as lower
troughs) than for the stable equilibria for smaller $\kappa$. 
Which scenario is more desirable or can be better tolerated is not a mathematical
question; it depends on factors such as the nature of the disease,
hardships at peak times, possibilities of intervention when the infection
ebbs, and so on. All we can offer is knowledge of which $\kappa$
will lead to what kinds of large-time dynamics for $I(t)$.

\section{Extension: latency time}

We discuss in this section a simple extension of the SIQ model, one that includes 
the idea of an latency period. This model divides the population into four groups, 
``S" for healthy and susceptible, ``E" for exposed but not yet infectious, ``I" for infectious, and ``Q" for isolation. The only change in the dynamics is as follows: Suppose an individual from Group S gets infected at time $t$. He enters Group E immediately and remains there for $\sigma$ units of time, $\sigma \ge 0$ being
a constant we will refer to as the {\it latency period}. 
For simplicity we assume that while in Group E, the individual is neither infectious 
(so he can infect no one),
nor symptomatic (so he cannot be identified and isolated), nor does recovery begin.
At time $t+\sigma$, he becomes infectious, enters Group I, and from this point on,
the rules for
identification, isolation, and recovery are the same as before. The parameters
in this extended model, which we call SEIQ, are $r, \sigma, p, \tau$ and $\kappa$. 
A schematic is shown in Fig. 8.1.

\begin{figure}[h]
\includegraphics[width=\linewidth]{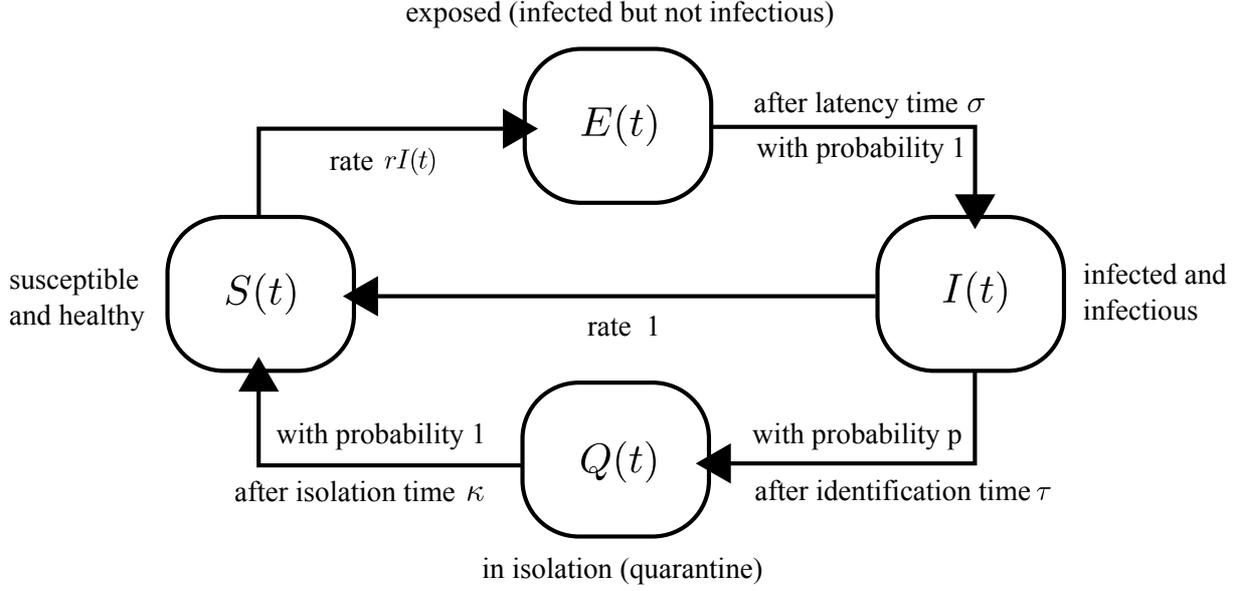}%
\caption{\label{fig:model-sigma}Illustration of the SEIQ model. 
Extension of the SIQ model with the additional feature that once infected (here referred to as $E(t)$ for exposed), individuals undergo an latency period of length $\sigma$ before becoming infectious.}
\end{figure}

We will follow a line of analysis similar to that in Secs. 3--7.
Note how the structures of SIQ persist and extend to the case with latency. 


\subsection{Mathematical set up and the set of equilibrium points}

First we write down the corresponding system of Delay Differential Equations,
derived in the same way:
\begin{eqnarray}
\dot{S}\negmedspace\left(t\right) & = & -rS\negmedspace\left(t\right)I\negmedspace\left(t\right)+I\negmedspace\left(t\right)+r\varepsilon S\negmedspace\left(t-\sigma-\tau-\kappa\right)I\negmedspace
\left(t-\sigma-\tau-\kappa\right),\label{eq:S-dyn-sigma}\\
\dot{E}\negmedspace\left(t\right) & = & rS\negmedspace\left(t\right)I\negmedspace\left(t\right)-rS\negmedspace\left(t-\sigma\right)I\negmedspace\left(t-\sigma\right),\label{eq:E-dyn-sigma}\\
\dot{I}\negmedspace\left(t\right) & = & rS\negmedspace\left(t-\sigma\right)I\negmedspace\left(t-\sigma\right)-I\negmedspace\left(t\right)-r\varepsilon S\negmedspace\left(t-\sigma-\tau\right)I\negmedspace\left(t-\sigma-\tau\right),\label{eq:I-dyn-sigma}\\
\dot{Q}\negmedspace\left(t\right) & = & r\varepsilon\left[S\negmedspace\left(t-\sigma-\tau\right)I\negmedspace
\left(t-\sigma-\tau\right)-S\negmedspace\left(t-\sigma-\tau-\kappa\right)I\negmedspace
\left(t-\sigma-\tau-\kappa\right)\right].\label{eq:Q-dyn-sigma}
\end{eqnarray}

As before, this system defines a $C^1$ semi-flow on the Banach space
\[
\mathcal{C}^\ast:=\{\phi\in C([-\sigma-\tau-\kappa,0],\mathbb{R}^{4})~| ~\phi_S(\theta)+\phi_E(\theta)+\phi_I(\theta)+\phi_Q(\theta)=1\mbox{ for all }\theta\in[-\sigma-\tau-\kappa,0]\}
,\]
equipped with the supremum norm. Whenever possible, we will use the same
notation, with an asterisk to distinguish it from the corresponding object
in the SIQ model. As before, we will study the dynamical system
on its full phase space $\mathcal{C}^\ast$, while paying special attention
to biologically relevant solutions, i.e., those with the property that for each $t\geq-\sigma-\tau-\kappa$, $x(t;\phi)=(S(t), E(t), I(t), Q(t))$ lies in the $3$-simplex 
\[
\Delta^{3}=\left\{ u=(u_{1},u_{2},u_{3},u_{4})\in\mathbb{R}^{4}:\sum_{i}u_{i}=1,u_{i}\ge0,i=1,2,3,4\right\}\ ,
\]
i.e., $x_{t}(\phi)\in\tilde{\mathcal{C}}^\ast=\{\psi\in\mathcal{C^\ast}:\psi(\theta)\in\Delta^{3}\mbox{ for all }\theta\in\left[-\sigma -\tau-\kappa,0\right]\}$ for all $t\ge0$. 

The simplest dynamical objects are equilibria. Noting that all of their coordinate functions are constant functions, one solves for them easily from 
Eqs.~\eqref{eq:S-dyn-sigma}--\eqref{eq:Q-dyn-sigma}. 
As before, we distinguish between the set of {\it disease-free equilibria}, defined by
\[
\mathcal{E}^{\ast}_{0}:=\left\{ \hat{\phi}\in\mathcal{C}^\ast~|~\,\phi_{I}=0 \right\} \qquad\mbox{and}\qquad\tilde{\mathcal{E}}^\ast_{0}:=\mathcal{E}
^\ast_{0}\bigcap\tilde{\mathcal{C}}^\ast,\]
and the set of {\it endemic equilibria}, characterized by $\phi_{I}\ne0$ and given by
\[
\mathcal{E}_{I}^\ast:=\left\{ \hat{\phi}\in\mathcal{C}^\ast|\,\phi_{S}=\frac{1}{r\left(1-\varepsilon\right)}\right\} \qquad\mbox{and}\qquad\tilde{\mathcal{E}}_{I}^\ast
:=\mathcal{E}_{I}^\ast\bigcap\tilde{\mathcal{C}^\ast}.
\]
Here $\mathcal{E}^{\ast}_{0}$ and $\mathcal{E}_{I}^\ast$ are 2D spaces 
with 
$$\mathcal{E}^{\ast}_{0}=\{\widehat{u(\eta, q)}: u(\eta, q)=(1-\eta-q, \eta, 0, q)\}$$
$$
\mbox{and} \qquad
\mathcal{E}_{I}^\ast =\{\widehat{w(\eta,q)}: w(\eta,q)=(1-q_c, \eta, q_c-\eta-q, q)\}\ .
$$

\subsection{Small outbreaks and critical response}

As in the case with $\sigma=0$, we analyze the stability of the equilibria in 
$\mathcal{E}^{\ast}_{0}$ for $\sigma>0$.

\begin{thm}
\label{thm:E0-sigma} Let $\tau$ and $p$ be fixed, and assume $\varepsilon=pe^{-\tau}<1$.
We denote
\[
q_{c}=1-\frac{1}{r}\frac{1}{1-\varepsilon}.
\] 
If \textup{$\eta+q\ge q_{c}$}, then \textup{$\widehat{u(\eta,q)}$} is linearly
stable, and if \textup{$\eta+q < q_{c}$}, then \textup{$\widehat{u(\eta,q)}$} is linearly unstable. In more detail, at $\eta+q> q_{c}$,
the eigenvalue $\lambda=0$ of the equilibrium $\widehat{u(\eta,q)}$ has
multiplicity $2$, and there is no other eigenvalue on the imaginary
axis. At $\eta+q= q_{c}$, a third zero eigenvalue $\lambda=0$ crosses the imaginary axis with nonzero speed, so that for $\eta+q<q_{c}$, there is at least one eigenvalue
$\lambda_{1}$ with $Re(\lambda_{1})>0$. 
\end{thm}

Theorem~\ref{thm:E0-sigma} generalizes Theorem~\ref{thm:E0} for $\sigma>0$. 
We remark that the stability boundary $q_c$ is the same as before (hence
 the same notation). A major difference here is that
for $\sigma>0$, we could not rule out further destabilizing bifurcations 
in the region $\eta+q<q_c$.

\begin{proof}
Linearizing Eqs.~\eqref{eq:S-dyn-sigma}--\eqref{eq:Q-dyn-sigma} around $\phi=\widehat{u(\eta,q)}$ reveals the characteristic equation 
$$ \lambda^2\left(\lambda+1-r(1-q-\eta)e^{-\sigma\lambda}(1-\varepsilon e^{-\tau\lambda})\right)=0.$$
Note that this equation is again independent of $\kappa$. There are two trivial eigenvalues corresponding to the directions $\eta,q$ along the $2$-parameter family $\tilde{\mathcal{E}}^{\ast}_{0}$. For the remaining part of the spectrum, we impose the ansatz $\lambda=i\omega$, $\omega>0$ to reveal potential bifurcation points. Note that we have already independently investigated the case $\omega=0$ in the proof of Theorem~\ref{thm:E0}. It holds that
\begin{equation}
\omega^2 = r^2(1-\eta-q)^2(1-2\varepsilon\cos(\omega\tau))-1 \leq  r^2(1-\eta-q)^2(1+2\varepsilon)-1,
\end{equation}
independently of $\sigma$. If the upper bound for $\omega^2$ is negative, then obviously no further bifurcation can occur. Now, $r^2(1-\eta-q)^2(1+2\varepsilon)-1<0$ implies
$$
\eta+q > 1-\frac{1}{r}\frac{1}{1+2\varepsilon} > q_c
$$
such that for $\eta+q>q_c$ there are no bifurcation points. The rest of the assertion follows directly from Theorem~\ref{thm:E0}.
\end{proof} 

As in Sec. 5.2, nonlinear theory applies to initial conditions in 
small neighborhoods of  $\mathcal{E}^{\ast}_{0}$: a small outbreak near 
$\phi=\widehat{u(\eta,q)}$ with $\eta+q > q_c$ is squashed, and one starting near 
$\eta+q < q_c$ will grow.

\medskip \noindent
{\bf Biological implications:}
Since it slows down the spreading of the disease, one might expect an infection with longer latency time to tolerate weaker
responses, e.g. a larger $\tau$. The analysis above shows otherwise: Consider an initial condition $\phi$ 
near $(\hat 1, \hat 0, \hat 0, \hat 0)$ with $\phi_I>0$. Since $q_c$ for $\sigma>0$
is identical to that for $\sigma=0$, it follows that the minimum isolation probability 
$p_c$ and critical delays $\tau_c(p)$ for each $p>p_c$ are all entirely independent 
of the latency period $\sigma$.
This can be understood as follows: In the case of no isolation, whether
or not a disease spreads has to do with the number of secondary cases,
referring to the number of individuals infected by a single infected
individual. This clearly has nothing to do with latency time.
With isolation, the same holds true, and the response is what
is done to decrease the number of secondary cases after an individual
becomes infectious, and that again has nothing to do with latency
time.

\subsection{Potential endemic equilibria for $\boldsymbol\sigma\boldsymbol>\boldsymbol0$}\label{subsec:endemic-sigma}

As with the case $\sigma=0$, Eqs.~\eqref{eq:S-dyn-sigma}--\eqref{eq:Q-dyn-sigma} possess conserved quantities in addition to the conservation of mass. 
Let $r,\kappa$ and $\sigma$
be fixed. We define $H_1^\ast=H_1^{r,\kappa,\sigma}:\mathcal{C}^\ast\to\mathbb{R}$ by 
\begin{equation}
H^\ast_1\left(\phi\right):=1-\phi_{S}(0)-\phi_{E}(0)-\phi_{I}(-\kappa)+\negmedspace \int\limits _{-\kappa}^{0}\phi_{I}(s)\mbox{d}s-r\negmedspace \int\limits _{-\sigma-\kappa}^{-\sigma}\phi_{S}(s)\phi_{I}(s)\mbox{d}s, \label{eq:H1-sigma}
\end{equation}
where $\phi=(\phi_{S},\phi_{S},\phi_{I},\phi_{Q})$. It is easy to see that $H_1^\ast$ is a natural generalization of $H$ for $\sigma>0$, as $H_1^{r,\kappa,0}=H^{r,\kappa}$, i.e. their values coincide, if $\sigma=0$ and $\phi_E(0)=0$. Define also
$H_2^\ast=H_2^{r,\kappa,\sigma}:\mathcal{C}^\ast\to\mathbb{R}$ by
\begin{equation}
H^\ast_2\left(\phi\right):=\phi_{E}(0)-r\negmedspace \int\limits _{-\sigma}^{0}\phi_{S}(s)\phi_{I}(s)\mbox{d}s, \label{eq:H2-sigma}
\end{equation} 
and let $H^\ast(\phi):=(H_1^\ast(\phi),H_2^\ast(\phi))$.

\begin{prop}
\label{prop:foliation}
For each fixed $r,\sigma,\kappa$, we have
\[
\frac{d}{dt} H^\ast (x_{t}(\phi))=0\qquad\mbox{for all } \phi\in\mathcal{C^\ast}\mbox{ and }t\ge0,
\]
and the level sets of $H^\ast$ define a smooth codimension 2 foliation
on $\mathcal{C^\ast}$. 
\end{prop}

The proof is analogous to the proof of Lemma~\ref{prop:foliation}. We leave it to the reader to check that the range of $(D_\phi H^\ast)$ is $2$-dimensional for any $\phi$. A suitable basis of the image is given by $\{\psi^1,\psi^2\}$, where $\psi^1_E=\hat{0}$ and $\psi^1_S,\psi^1_I,\psi_Q$ are defined as in the proof of Lemma~\ref{prop:foliation}, and $\psi^2=(\hat{0},\hat{1},\hat{0},\hat{0})$.

Let $r, \kappa, \sigma$ be fixed. We let 
$\mathcal{F}^\ast = \mathcal{F}^{r,\kappa,\sigma}$ denote the foliation defined by 
$H^{r,\kappa,\sigma}$, and let 
$\mathcal{F}_{\eta,q}^\ast$
denote the leaf of $\mathcal{F}^\ast$ containing the point 
$\widehat{u(\eta,q)}\in\mathcal{\mathcal{E}}^{\ast}_{0}$. 
The following is the analog of Theorem 7.

\begin{thm}
\label{thm:candidate-sigma} For fixed $r,\kappa,\sigma,\eta,q,$ we let $\mathcal{F}^\ast=\mathcal{F}^{r,\kappa,\sigma},$ and consider $\mathcal{F}^\ast_{\eta,q}$. 
\begin{enumerate}
\item[(1)] $\mathcal{F}_{\eta,q}^\ast\cap\mathcal{E}^\ast_{0}=\{ \widehat{u(\eta,q)} \}$, where $u(\eta,q)=(1-\eta-q,\eta,0,q)$. 
\item[(2)] Fixing additionally $p,\tau$, which determines $\mathcal{E}^\ast_{I}=\{ \phi_S=[r(1-\varepsilon)]^{-1} \}$, we have $ \mathcal{F}_{\eta,q}^\ast \cap \mathcal{E}^\ast_{I}= \{ \widehat{v(\eta,q)} \}$, where $v(\eta,q)=(v_S,v_E(\eta,q),v_I(\eta,q),v_Q(\eta,q))$ and
\begin{eqnarray}
v_{E}(\eta,q) &=&
\frac{\sigma}{1-\varepsilon+\sigma+\varepsilon\kappa}\left(q_{c}-q-\eta\right)+\eta, \nonumber \\
v_{I}(\eta,q) &= & \frac{1-\varepsilon}{1-\varepsilon+\sigma+\varepsilon\kappa}\left(q_{c}-q-\eta\right), \nonumber \\
v_{Q}(\eta,q) &=&  \frac{\varepsilon\kappa}{1-\varepsilon+\sigma+\varepsilon\kappa}\left(q_{c}-q-\eta\right)+q.\nonumber
\end{eqnarray}
\end{enumerate}
\end{thm}

The proof of Theorem 12 follows from straightforward computation. Part (2) relies on solving 
the system of equations
\begin{eqnarray}
q &=& H_1^\ast(\widehat{w(\eta',q')}), \nonumber \\
\eta &=& H_2^\ast(\widehat{w(\eta',q')}), \nonumber\ 
\end{eqnarray}
where $w$ is defined as at the end of Sec. 8.1. 
From the form of $H_1^\ast$ and $H_2^\ast$, one sees that the quantities
on the right are linear combinations of $\eta'$ and $q'$. 

\medskip \noindent
{\bf Biological implications:} When a small outbreak occurs and the response
(in terms of $p$ and $\tau$) is inadequate, the infection will spread. In our model,
this corresponds to starting from an initial condition near an unstable disease-free
equilibrium point $\widehat{u(\eta,q)}$ for some $0 \le \eta, q \le 1$ with 
$\eta+q < q_c$. Such an infection may eventually approach an endemic 
equilibrium, or it may fluctuate indefinitely. 
Theorem 12 tells us that there is a unique endemic equilibrium to which
it can potentially converge, and predicts the fraction of infected individuals
in this endemic equilibrium. 

Note that the latency period $\sigma$ does appear in the formula for $v_I$; 
the longer the latency, the smaller
the fraction of infectious individuals.
Note also that $\sigma$ and $\varepsilon \kappa$ play similar roles in the formulas 
in Theorem 12: both involve taking subpopulations
out of circulation, so they neither infect nor can be further infected. There is a pre-factor $\varepsilon$ in front of $\kappa$ as $1-\varepsilon$ represents the degree to which the isolation procedure is compromised. 


\section{Outlook}

We have investigated how a simple isolation scheme can affect the long-term
dynamics of an infection.
We have identified a critical isolation probability $p_c$ and a critical 
identification time $\tau_{c}$, and have proved that the infection cannot persist if
one has the capability to isolate sufficiently many hosts within a
sufficiently short time after a host's infection. Moreover, we have
carefully investigated how the length of isolation affects the outcome
of an epidemic if these thresholds are not met, and have found,
a little counterintuitively, that longer isolation can lead to oscillations 
in the fraction of infected hosts that periodically rises above that for 
shorter lengths of isolation. 

Our underlying model is, of course, highly idealized, and needs to be
modified substantially before it can be applied to real world scenarios.
To demonstrate that it offers a clear and promising starting point for the systematic
analysis of isolation processes, we investigated a first extension of the model,
to the case where infected hosts do not become infectious directly after exposure with the disease but undergo an latency period. We showed in this extended model
that our results
for the original model are not changed substantially, and the structures
are robust.

One of the most serious simplifications in the work presented is that 
we have neglected the underlying spatial topology of the infection process. 
We recognize its impact on disease evolution,
and the need to  incorporate network heterogeneity in future work. Well known techniques include higher order moment closure techniques as suggested in \cite{Kiss2015}, and heterogeneous mean field approximations \cite{Barthelemy2005}. For many problems temporal networks with adaptive wiring can be useful \cite{Belik}. Other steps towards realism 
include the incorporation of basic disease characteristics and
data-driven modeling, which has become increasingly feasible thanks 
to modern mobile technologies capable of reporting relevant data 
in real time  \cite{Salath??2012,Stopczynski2014}.

\subsection*{Acknowledgment}
The authors would like to thank Odo Diekmann
and Dimitry Turaev for critical discussion, and Stefan Ruschel would
like to thank the University of São Paulo in São Carlos and New York
University for their hospitality. 
This paper was developed within the scope of the IRTG 1740/ TRP 2015/50122-0,
funded by the DFG/FAPESP. Lai-Sang Young was partially supported by
NSF Grant DMS-1363161. Tiago Pereira was partially supported by  FAPESP grant 2013/07375-0. 
\section*{Appendix \label{sec:Appendix}}

\subsection*{Description of SIQ network model. }

\noindent We consider an undirected, unweighted and stationary (contact)
network with $N$ nodes and average degree $\left\langle k\right\rangle $,
and an infection spreading process on this network as treated in \cite{Pereira2015}.
Specifically, each node can be susceptible $\left(S\right)$, infected
$\left(I\right)$ or isolated $\left(Q\right)$. A susceptible
node is infected by each one of its infected neighbors at rate $\beta$.
Infected nodes recover, i.e., revert to the susceptible state, at
rate $\gamma$. Additionally, nodes that remain infected for time
$\tau$ are isolated with probability $p$. Isolated nodes cannot
be infected by any of their neighbors and do not infect susceptible
neighbors. We augment \cite{Pereira2015} to allow for finite times
of isolation, assuming that each node in isolation is discharged after
time $\kappa$. Upon discharge a node is immediately susceptible again
and retains the same neighborhood as prior to isolation.

\subsection*{Modeling the SIQ network by systems of delay differential equations}

To approximate the epidemic spreading process described above, we
follow \cite{Keeling1999,Kiss2015} in spirit and in notation. Let
$\left[S\right]\negmedspace\left(t\right),\left[I\right]\negmedspace\left(t\right),\left[Q\right]\negmedspace\left(t\right)$
denote the number of susceptible, infected, and isolated nodes
respectively at time $t$, and $\left[SI\right]\negmedspace\left(t\right)$
the number of links between susceptible and infected nodes. For $A\in\left\{ S,I,Q\right\} $,
we use $\left[\textrm{\ensuremath{\rightarrow}}A\right]\negmedspace\left(t,0\right)$
to denote the rate at which nodes enter state ($A$) at time $t$,
and for $s>0$, we define 
\[
\left[\textrm{\ensuremath{\rightarrow}}A\right]\negmedspace\left(t,s\right):=\lim_{\Delta t\to0}\frac{1}{\Delta t}\left(\begin{array}{c}
\mbox{\# nodes entering (\ensuremath{A}) on \ensuremath{\left[t,t+\Delta t\right]}}\\
\mbox{ and remaining till time \ensuremath{t+s}}
\end{array}\right)\ .
\]
Note that $t+s\in\mathbb{R}$ is the time point at which we observe
$\left[\textrm{\ensuremath{\rightarrow}}A\right]\negmedspace\left(t,s\right)$.
In the mathematical biology literature, one often divides the population
into cohorts. In our model, for $s>\tau$ the quantity $\left[\textrm{\ensuremath{\rightarrow}}I\right]\negmedspace\left(t-s,s\right)$
represents the size of the cohort of nodes newly infected at time
$t-s$, did not enter isolation at time $t-s+\tau$, and have not
recovered by time $t$.

We infer the total number of nodes in state $\left(A\right)$
at a given time with the help of these quantities. From the network
process, we have that $\left[\textrm{\ensuremath{\rightarrow}}I\right]\negmedspace(t,0)=\beta\left[SI\right]\negmedspace\left(t\right)$.
Similarly, we have that $\left[\textrm{\ensuremath{\rightarrow}}Q\right]\negmedspace(t,0)=\beta\left[SI\right]\negmedspace\left(t-\tau\right)\cdot e^{-\gamma\tau}\cdot p$,
i.e., with probability $p$, nodes infected at time $t-\tau$ that
have not recovered by time $t$ will enter isolation at this time.
Incorporating the rules for isolation, we obtain the simple relations
\[
\left[\textrm{\ensuremath{\rightarrow}}I\right]\negmedspace(t-s,s)=\beta e^{-\gamma s}\left[SI\right]\negmedspace\left(t-s\right)\left(1-pH\left(s-\tau\right)\right),
\]
\[
\left[\textrm{\ensuremath{\rightarrow}}Q\right]\negmedspace(t-s,s)=\beta pe^{-\gamma\tau}\left[SI\right]\negmedspace\left(t-s-\tau\right)\left(1-H\left(s-\kappa\right)\right)
\]
where $H(\cdot)$ is the heaviside function with $H(s)=1$ for $s\ge0$,
and $H(s)=0$ for $s<0$.

Given an initial condition $(\left[S\right]\negmedspace\left(t\right),\left[I\right]\negmedspace\left(t\right),\left[Q\right]\negmedspace\left(t\right))$
where $\left[S\right]\negmedspace\left(t\right),\left[I\right]\negmedspace\left(t\right)$
and $\left[Q\right]\negmedspace\left(t\right)$ are continuous functions
defined on the interval $[-\tau-\kappa,0]$, it is not hard to deduce
from the relations above that for $t\ge0$, 
\begin{eqnarray*}
\left[S\right]^{\prime}\negmedspace\left(t\right) & = & -\beta\left[SI\right]\negmedspace\left(t\right)+\gamma\left[I\right]\negmedspace\left(t\right)+\left[\textrm{\ensuremath{\rightarrow}}Q\right]\negmedspace\left(t-\kappa,\kappa\right),\\
\left[I\right]^{\prime}\negmedspace\left(t\right) & = & \beta\left[SI\right]\negmedspace\left(t\right)-\gamma\left[I\right]\negmedspace\left(t\right)-p\left[\textrm{\ensuremath{\rightarrow}}I\right]\negmedspace\left(t-\tau,\tau\right),\\
\left[Q\right]^{\prime}\negmedspace\left(t\right) & = & p\left[\textrm{\ensuremath{\rightarrow}}I\right]\negmedspace\left(t-\tau,t\right)-\left[\textrm{\ensuremath{\rightarrow}}Q\right]\negmedspace\left(t-\kappa,\kappa\right)\ .
\end{eqnarray*}
The equations for $\left[I\right]^{\prime}\negmedspace\left(t\right)$
and $\left[Q\right]^{\prime}\negmedspace\left(t\right)$ are obtained
by direct computation of derivatives, and the one for $\left[S\right]^{\prime}\negmedspace\left(t\right)$
is obtained by setting $\left[S\right]\negmedspace\left(t\right)+\left[I\right]\negmedspace\left(t\right)+\left[Q\right]\negmedspace\left(t\right)=N$.

Closing this model as proposed in \cite{Keeling1999} by the approximation
$\left[SI\right]\negmedspace\left(t\right)\approx\left\langle k\right\rangle \frac{\left[S\right]\negmedspace\left(t\right)}{N}\left[I\right]\negmedspace\left(t\right)$,
that is, by neglecting any correlation between $\left(S\right)$ and
$\left(I\right)$ nodes, and assuming for now (we will return to this
point later) that the relations $\left[\textrm{\ensuremath{\rightarrow}}I\right]\negmedspace(t-s,s)$
and $\left[\textrm{\ensuremath{\rightarrow}}Q\right]\negmedspace(t-s,s)$
above hold for all $t-s\ge-\tau-\kappa$, we obtain 
\begin{eqnarray*}
\left[S\right]^{\prime}\negmedspace\left(t\right) & = & -\beta\frac{\left\langle k\right\rangle }{N}\left[S\right]\negmedspace\left(t\right)\left[I\right]\negmedspace\left(t\right)+\gamma\left[I\right]\negmedspace\left(t\right)+\beta pe^{-\gamma\tau}\frac{\left\langle k\right\rangle }{N}\left[S\right]\negmedspace\left(t-\tau-\kappa\right)\left[I\right]\negmedspace\left(t-\tau-\kappa\right),\\
\left[I\right]^{\prime}\negmedspace\left(t\right) & = & \beta\frac{\left\langle k\right\rangle }{N}\left[S\right]\negmedspace\left(t\right)\left[I\right]\negmedspace\left(t\right)-\gamma\left[I\right]\negmedspace\left(t\right)-\beta pe^{-\gamma\tau}\frac{\left\langle k\right\rangle }{N}\left[S\right]\negmedspace\left(t-\tau\right)\left[I\right]\negmedspace\left(t-\tau\right),\\
\left[Q\right]^{\prime}\negmedspace\left(t\right) & = & \beta pe^{-\gamma\tau}\frac{\left\langle k\right\rangle }{N}\left(\left[S\right]\negmedspace\left(t-\tau\right)\left[I\right]\negmedspace\left(t-\tau\right)-\left[S\right]\negmedspace\left(t-\tau-\kappa\right)\left[I\right]\negmedspace\left(t-\tau-\kappa\right)\right).
\end{eqnarray*}

Finally, we rescale the state variables $S(t)=[S](t)/N,\ I(t)=[I](t)/N,\ Q(t)=[Q](t)/N$,
rescale time $\tilde{t}=t/\gamma$, write $\frac{\textrm{d}}{\textrm{d}\tilde{t}}=\dot{}$\ ,
and introduce the rescaled parameters $r=\frac{\beta\left\langle k\right\rangle }{\gamma},\:\varepsilon=pe^{-\gamma\tau},\:\tilde{\kappa}=\gamma\kappa,\:\tilde{\tau}=\gamma\tau$
to obtain 
\begin{eqnarray*}
\dot{S}\negmedspace\left(\tilde{t}\right) & = & -rS\negmedspace\left(\tilde{t}\right)I\negmedspace\left(\tilde{t}\right)+I\negmedspace\left(\tilde{t}\right)+r\varepsilon S\negmedspace\left(\tilde{t}-\tilde{\tau}-\tilde{\kappa}\right)I\negmedspace\left(\tilde{t}-\tilde{\tau}-\tilde{\kappa}\right),\\
\dot{I}\negmedspace\left(\tilde{t}\right) & = & rS\negmedspace\left(\tilde{t}\right)I\negmedspace\left(\tilde{t}\right)-I\negmedspace\left(\tilde{t}\right)-r\varepsilon S\negmedspace\left(\tilde{t}-\tilde{\tau}\right)I\negmedspace\left(\tilde{t}-\tilde{\tau}\right),\\
\dot{Q}\negmedspace\left(\tilde{t}\right) & = & r\varepsilon\left(S\negmedspace\left(\tilde{t}-\tilde{\tau}\right)I\negmedspace\left(\tilde{t}-\tilde{\tau}\right)-S\negmedspace\left(\tilde{t}-\tilde{\tau}-\tilde{\kappa}\right)I\negmedspace\left(\tilde{t}-\tilde{\tau}-\tilde{\kappa}\right)\right).
\end{eqnarray*}
For notational simplicity, we omit the tildes from here on, but it
is important to keep in mind that our findings are stated in the characteristic
timescale of the recovery process. These are Eqs.~\eqref{eq:S-dyn}\textendash \eqref{eq:Q-dyn}
in the main text.

\subsection*{Positivity of solutions}

Since Eqs.~\eqref{eq:S-dyn}\textendash \eqref{eq:Q-dyn}
are intended to describe transfer of mass among the states $(S),(I)$
and $(Q)$, one might expect that they satisfy not only mass conservation,
i.e. $S(t)+I(t)+Q(t)\equiv1$, but also positivity, i.e., $S(t),I(t),Q(t)\ge0$,
for all $t\ge0$, provided these conditions are satisfied by the initial
condition. The positivity part, however, is not true without further
assumptions as we now explain. Let $\left[I\right]\!\left(t\right),\left[Q\right]\!\left(t\right),$
$t\in\left[-\tau-\kappa,0\right],$ be given. Then for $t\geq0,$
we may split $\left[I\right]\!\left(t\right)$ into 
\[
\left[I\right]\!\left(t\right):=\left[I\right]_{1}\!\left(t\right)+\left[I\right]_{2}\!\left(t\right),
\]
where $\left[I\right]_{1}\!\left(t\right)$ and $\left[I\right]_{2}\!\left(t\right)$
represent the contribution to the number of nodes entering before
and after time $0$ respectively. We then have 
\begin{equation}
\left[I\right]_{1}\!\left(t\right)=e^{-\gamma\left(t\right)}\left(\left[I\right]\!\left(0\right)-\int\limits _{\max\left\{ -t+\tau,0\right\} }^{\tau}p\left[\rightarrow I\right]\negmedspace\left(-s,s\right)\textrm{d}s\right),\label{eq:reconstr}
\end{equation}
and 
\[
\left[I\right]_{2}\!\left(t\right)=\int\limits _{0}^{t}\left[\rightarrow I\right]\negmedspace\left(t-s,s\right)\textrm{d}s.
\]
The limits of integration in the integral in $\left[I\right]_{1}\!\left(t\right)$
are deduced from the fact that for $0\leq t\leq\tau$, nodes that
leave ($I$) for ($Q$) on $\left[t_{0},t\right]$ entered on the
time interval $\left[-\tau,t-\tau\right]$, whereas for $t\geq\tau$,
these nodes entered on the time interval $\left[-\tau,0\right]$.

In the derivation of the delay equations above, we have assumed that
$[\rightarrow I]\negmedspace\left(-s,0\right)$ is proportional to
$\left[S\right]\!\left(-s\right)\left[I\right]\!\left(-s\right)$
and recovery occurs at rate $\gamma$, but this need not be true in
the given initial condition: it can happen that the number of nodes
in state $(I)$ for $t<0$ is smaller than assumed. Such discrepancies
can result in $\left[I\right]_{1}\!\left(t\right)<0$ when we transfer
more mass out of $(I)$ than is actually present.

This is the only way $\left[I\right]\!\left(t\right)$ can become
negative. That is to say, $\left[I\right]\!\left(t\right)$ is guaranteed
to be non-negative for all $\ge0$ for initial conditions for which
$\left[I\right]\!\left(0\right)\ge p\int_{0}^{\tau}\left[\rightarrow I\right]\negmedspace\left(-s,s\right)\textrm{d}s$. A similar analysis holds for $\left[Q\right]\!\left(t\right)$.
These results are recorded in Lemma 1 in Sec. 3.3.

\bibliographystyle{unsrt}
\bibliography{EpiNet, EpiNet-old}

\end{document}